\documentclass[preprint,12pt]{article}

\usepackage{a4wide}
\usepackage{graphicx}
\usepackage{subcaption}
\usepackage{amssymb,amsmath,amsthm,mathrsfs,xfrac}
\usepackage{enumerate,color}
\usepackage{hyperref}
\usepackage{array}
\usepackage{amsfonts}
\usepackage{psfrag}
\usepackage{pgf,tikz}
\usetikzlibrary{arrows}
\usepackage{ulem}
\usepackage{url}
\usepackage[title]{appendix}

\newtheorem{theorem}{Theorem}
\newtheorem{assumption}{Assumption}
\newtheorem{cor}[theorem]{Corollary}
\newtheorem{definition}{Definition}
\newtheorem{remark}{Remark}

\newtheorem{lemma}{Lemma}
\newtheorem{proposition}{Proposition}

\numberwithin{equation}{section}
\allowdisplaybreaks[4]

\renewenvironment{proof}{\smallskip\noindent\emph{\textbf{Proof.}}%
  \hspace{1pt}}{\hspace{-5pt}{\nobreak\quad\nobreak\hfill\nobreak%
    $\square$\vspace{2pt}\par}\smallskip\goodbreak}

\newenvironment{proofof}[1]{\smallskip\noindent{\textbf{Proof~of~#1.}}%
  \hspace{1pt}}{\hspace{-5pt}{\nobreak\quad\nobreak\hfill\nobreak%
    $\square$\vspace{2pt}\par}\smallskip\goodbreak}

\newcommand{\C}[1]{\mathbf{C^{#1}}}
\newcommand{\Cc}[1]{\mathbf{C}_c^{#1}}

\newcommand{\BV}{\mathbf{BV}}
\renewcommand{\L}[1]{\mathbf{L^#1}}
\newcommand{\Lloc}[1]{\mathbf{L^{#1}_{loc}}}

\newcommand{\modulo}[1]{{\left|#1\right|}}
\newcommand{\norma}[1]{{\left\|#1\right\|}}

\newcommand{\R}{\mathbb{R}}

\newcommand{\N}{\mathbb{N}}

\newcommand{\Z}{\mathbb{Z}}

\renewcommand{\epsilon}{\varepsilon}
\renewcommand{\phi}{\varphi}
\renewcommand{\theta}{\vartheta}

\newcommand{\tv}{\mathinner{\rm TV}}

\renewcommand{\d}[1]{\mathinner{\mathrm{d}{#1}}}

\newcommand{\brho}{\boldsymbol{\rho}}

\newcommand{\del}{\partial}
\newcommand{\be}{\begin{equation}}
\newcommand{\ee}{\end{equation}}

\definecolor{ffqqqq}{rgb}{1.,0.,0.}
\definecolor{uuuuuu}{rgb}{0.26666666666666666,0.26666666666666666,0.26666666666666666}

\graphicspath{ {./images/} }

\DeclareMathOperator{\sgn}{sgn}

\makeatletter
\let\@fnsymbol\@arabic
\makeatother

\makeatletter
\newcommand*\xfootnote[1][*]{%
	\xdef\@thefnmark{#1}%
	\@footnotemark\@footnotetext}
\makeatother

\title{ 
A multi-class non-local macroscopic model with time delay for mixed autonomous / human-driven traffic}

\author{
\textsc{Ilaria Ciaramaglia\footnotemark[0]\;\;\footnotemark[1]}
\and
\textsc{Paola Goatin\footnotemark[1]}
\and
\textsc{Gabriella Puppo\footnotemark[2]}
}

\date{\today}

\begin{document}
\maketitle

\footnotetext[1]{ Universit\'e C\^ote d'Azur, Inria, CNRS, LJAD, 2004 route des
 Lucioles - BP 93, 06902 Sophia Antipolis Cedex, France. E-mail:
 \texttt{\{ilaria.ciaramaglia, paola.goatin\}@inria.fr}}
 
 \footnotetext[2]{Dipartimento di Matematica - Sapienza, Università di Roma; P.le Aldo Moro, 5 - 00185
Roma, Italy. E-mail: \texttt{gabriella.puppo@uniroma1.it}}

\footnotetext{Corresponding author: Ilaria Ciaramaglia.}

\begin{abstract}
In this paper, we present a class of systems of non-local conservation laws in one space-dimension incorporating time delay, which can be used to investigate the interaction between autonomous and human-driven vehicles, each characterized by a different reaction time and interaction range. 
We construct approximate solutions using a Hilliges-Weidlich scheme and we provide uniform $\L\infty$ and $\BV$ estimates which ensure the convergence of the scheme, thus obtaining existence of entropy weak solutions of bounded variation. Uniqueness follows from an $\L 1$ stability result derived from the entropy condition. 
Additionally, we provide numerical simulations to illustrate applications to mixed autonomous / human-driven traffic flow modeling. 
In particular, we show that the presence of autonomous vehicles improves overall traffic flow and stability.

\medskip

  \noindent\textit{2020~Mathematics Subject Classification:} 35L65, 35L03,
  65M12, 76A30.

  \medskip

  
\end{abstract}

\section{Introduction}
During the last century, the automobile industry has accelerated its efforts in developing smart vehicles. Also known as self-driving cars or driverless cars, autonomous vehicles (AVs) are computer-controlled cars capable of making quick decisions according to their surroundings and guiding themselves without human interaction \cite{Hussain2019}. 

The origins of autonomous driving technique can be traced back to the 1920s with the \textit{phantom autos} that were cars controlled through a remote control device \cite{firsAVweb}. 
Nowadays, recent advancements in communication and self-driving technology have made connected autonomous vehicles (CAVs) a crucial element of urban transportation systems.
Compared to traditional human-driven cars, CAVs exchange data with other vehicles and infrastructure in real-time and with an almost instantaneous reaction \cite{PAN2024}. 

One of the potential benefits of intelligent transportation systems (ITS) $-$ frameworks integrating advanced communication, control, and information technologies into vehicles and infrastructure $-$ concerns road safety.
AVs have the potential to dramatically reduce accidents caused by human error, which is believed to be the leading cause behind $94\%$ of all crashes from $2005$ to $2007$ \cite{Singh2018}.
Self-driven vehicles would not be subject to human failings involving inexperience, slow reaction times, over-compensation and inattention. Data collected for the \textit{Global status report on road safety 2023 survey} \cite{WHOsafety2023} indicate that around $10\%$ of road-crash deaths are related to ``driving under the influence'' of alcohol or drugs, and such data do not reflect road fatalities due to speeding, distraction (e.g. use of mobile phone) or other prohibited driver maneuvers. According to \cite{FAGNANT2015}, an ITS has the potential for at least a $40\%$ fatal crash-rate reduction.

In addition to making automobiles safer, AV technology has the potential to improve traffic flow and reduce congestion and fuel consumption. This is due to the fact that AVs could enable quicker reaction times and closer spacing between vehicles to counteract increasing demand, thus leading to smoother traffic patterns \cite{multiscale2021,GHIASI2017266,GongPiccoliVisconti2021,HertyPuppoVisconti2023}.
Some experiments indicate that controlling a group of AVs into traditional human-driven traffic can significantly stabilize the flow. In particular, if AVs are evenly distributed in space, a penetration rate of around $5\%$ is sufficient to dampen stop-and-go waves \cite{AVsexperiments2018}. However, if the AVs are not evenly spaced on the roadway, then a higher penetration rate may be required to achieve the same wave-dampening effects.
In these studies, the typical setup involves human-driven cars along with a few AVs connected to an external control structure designed to improve stability. In contrast, the framework discussed in this paper does not rely on external controls.
In fact, the observed stabilizing effect is due merely to the characteristic dynamics of the AVs, based on a large look-ahead distance and a negligible time delay, which contribute to stabilizing the entire traffic mixture.

\subsection{Modeling}

Mathematical modeling plays a key role in understanding and improving ITS.
For our purposes, we characterize autonomous and human-driven vehicles (HVs) by their reaction time and interaction horizon. Indeed, we can assume that smart cars can collect information on the surrounding traffic within a large perimeter and are able to instantaneously react to external events, whereas human drivers react to downstream traffic in a shorter range and with a non-zero time delay. 
Following this idea, we propose the following macroscopic  model for mixed traffic composed of vehicles with different reaction times and interaction range. 

We consider $M$ classes of vehicles and we fix $M$ non-negative constants $\tau_1,\dots,\tau_M$, each representing the reaction time of the $i$-th population. The model consists in the $M\times M$ system of conservation laws
\begin{equation}\label{multiclasse}
\partial_t \rho_i(t,x)+\partial_x\left(\rho_i(t,x)f_i(\rho_i(t,x))v_i((r\ast \omega_i)(t-\tau_i,x))\right)=0,\qquad\forall i=1,\dots,M,
\end{equation}
where $\rho_i:\R^+\times\R\rightarrow[0,R_i]$ is the vehicle density  associated to the $i$-th class, $v_i:[0,+\infty)\rightarrow [0,V_i]$ is its mean traffic speed, $\omega_i:[0,L_i]\rightarrow\R^+$ its convolution kernel and $f_i:[0,R_i]\rightarrow [0,1]$ its saturation function. The positive constants $R_i,V_i,L_i$ are respectively the maximal density, the maximal speed and the look-ahead distance of drivers in the $i$-th population.

The saturation function $f_i$ indicates the free space on the road and it is necessary to guarantee the maximum principle, i.e. the fulfillment of the maximal density bound. We underline that the equations in the system~\eqref{multiclasse} are not coupled in the saturation term, but only in the non-local velocities, which depend on the total density $r=\rho_1+\dots+\rho_M$. 
This is essential to derive the well-posedness results contained in this paper. Nevertheless, 
from the modeling point of view,
it would make sense to also have saturation functions depending on the total density. This choice is investigated numerically in Section~\ref{saturationtotal}. 

In the sequel, we use the compact notation $\norma{\cdot}$ for $\norma{\cdot}_{\infty}$, and we introduce the vector $\boldsymbol\tau=\left(\tau_1,\dots,\tau_M\right)$ of the time delay parameters.
Due to the delayed time dependence, the system \eqref{multiclasse} needs to be coupled with an initial condition defined on the interval $[-\norma{\boldsymbol\tau},0]$, which reduces to a point if $\boldsymbol\tau=\left(0,\dots,0\right)$. Anyway, in certain modelling applications, it might only be possible to gather the traffic data at a certain initial time $t=0$. Thus, we couple \eqref{multiclasse} with an initial condition which is obtained as a constant backward extension of the data $\rho_i(0,x)=\rho_i^0(x)$, meaning that we assume 
\begin{equation}\label{eq:initial_datum}
\rho_i(t,x)=\rho_i^0(x),\qquad \forall (x, t)\in\left[-\norma{\boldsymbol\tau},0\right]\times\R,
\end{equation}
for all $i=1,\dots,M$.
For this particular choice of past-time data, the problem can be read as a classical Cauchy problem and  we will denote by $\boldsymbol\rho^0=\left(\rho^0_1,\dots,\rho^0_M\right)$ the vector containing the initial conditions. Due to the possible presence of jump discontinuities, solutions are intended in the following weak sense.

\begin{definition}[Weak solution]
    Given $\boldsymbol\rho^0\in\L1(\R;[0,R_1]\times\dots\times[0,R_M])$, a function $\boldsymbol\rho\in\L1([0,T]\times\R;\R^M)$ is a weak solution of the Cauchy problem \eqref{multiclasse}-\eqref{eq:initial_datum} if each component $\rho_i$, $i=1,\dots,M$, satisfies 
    \begin{align*}
    \int_0^T\int_\R &\left(\rho_i\del_t\phi+\rho_i f_i(\rho_i)v_i\left((r\ast\omega_i)(t-\tau_i,x)\right)\del_x\right)\d x\d t+\int_\R \rho_i^0(x)\phi(0,x)\d x=0,
\end{align*}
for every test function $\phi\in\Cc1([0,T)\times\R;\R^+)$.
\end{definition}

Anyway, as usual for hyperbolic systems, to guarantee the model well-posedness, solutions to~\eqref{multiclasse}-\eqref{eq:initial_datum} must be intended in the entropy weak sense \cite{Kruzkov}.

\begin{definition}[Entropy weak solution]\label{multientropy}
Given $\boldsymbol\rho^0\in\L1(\R;[0,R_1]\times\dots\times[0,R_M])$, a function $\boldsymbol\rho\in\L1([0,T]\times\R;\R^M)$ is an entropy weak solution of the Cauchy problem \eqref{multiclasse}-\eqref{eq:initial_datum} if each component $\rho_i$, $i=1,\dots,M$, satisfies 
\begin{align*}
    \int_0^T\int_\R &\big(|\rho_i-\kappa|\del_t\phi+\sgn(\rho_i-\kappa)\left(\rho_i f_i(\rho_i)-\kappa f_i(\kappa)\right)v_i\left((r\ast\omega_i)(t-\tau_i,x)\right)\del_x\phi\\
    &-\sgn(\rho_i-\kappa)\kappa f_i(\kappa)\partial_xv_i(r\ast\omega_i)(t-\tau_i,x)\phi\big)\d x\d t\\
    &+\int_\R |\rho_i^0(x)-\kappa|\phi(0,x)\d x\geq 0,
\end{align*}
for every test function $\phi\in\Cc1([0,T)\times\R;\R^+)$ and $\kappa\in\R $.
\end{definition}

\subsection{Literature review}

Multi-class traffic models were first proposed in~\cite{BenzoniColombo2003, wongwong2002} to account for the diversity in vehicle types, characteristics, and behaviors. 
 They are a natural extension of the Lighthill and Whitham \cite{LighthillWhitham} and Richards \cite{Richards} (LWR) model and are given by the system of conservation laws
\begin{equation}\label{npopulationmodel_Colombo}
\partial_t\rho_i+\partial_x(\rho_iv_i(r))=0,\qquad\forall i=1,\dots,M,
\end{equation}
where the average speed $v_i:[0,R]\rightarrow[0,V_i]$ depends on the mean free space between vehicles $r=l_1\rho_1+\dots+l_M\rho_M$, where $R$ is the maximum occupied space on the road and $l_i$ is the average length of the vehicles in the $i$-th class. By means of the rescaling $\rho_i\mapsto l_i\rho_i$, without loss of generality we can assume $l_i=1$ for all $i$, hence 
$$
r=\rho_1+\dots+\rho_M.
$$
Traffic heterogeneity can be rescaled by means of passenger car equivalent (PCE) values \cite{ChitturiBenekohal2007,fastlane2008,IngleHancockAlkaisy2007,WebsterElefteroadou1999}, that is \textit{``the number of passenger cars that are displaced by a single heavy vehicle of a particular type under prevailing roadway, traffic, and control conditions''} \cite{HighCapacityManual}.
Moreover, Fan and Work \cite{FanWork2015} also allow for class-specific maximal densities $R_i$, $i=1,\ldots,M$. This allows to incorporate the creeping behaviour, when smaller vehicles can move through
 larger vehicles that are stopped in congestion.

More recently, nonlocal conservation laws have been used to model nonlocal interactions in several applications, such as crowd dynamics~\cite{ACG2015} and opinion formation \cite{Piccoli2019}. 
Regarding traffic flow modeling, non-local versions of the Lighthill-Whitham-Richards model have been proposed in~\cite{BlandinGoatin2016, ChiarelloGoatin2018, 2018Gottlich}, where the speed function depends on a weighted mean of the downstream vehicle density, or is a weighted mean of the downstream velocities. 

From a modeling perspective, the non-locality in space allows for a more accurate representation of drivers adapting their velocity to the road condition within their visibility range. One would expect that taking into account traffic conditions downstream would improve traffic behavior. Interestingly, this translates mathematically into an increased stability of the solution~\cite{HuangDu2022}.
A multi-class extension was proposed in~\cite{ChiarelloGoatin2019} as
\begin{equation}\label{multiChiarello}
\partial_t\rho_i(t,x)+\partial_x(\rho_i(t,x)v_i((r\ast\omega_i)(t,x)))=0,\qquad\forall i=1,\dots,M,
\end{equation}
using the same notations as \eqref{multiclasse}.
 The model takes into consideration traffic heterogeneity characterizing the vehicles by their speed functions and look-ahead distances. Existence of solutions of~\eqref{multiChiarello} has been proved in \cite[Theorem 1.2]{ChiarelloGoatin2019} only locally in time due to the blow-up of the $\L\infty$ bounds. We recall that global existence results for general systems of non-local conservation laws in several space dimensions were provided by~\cite{ACG2015}, requiring smooth convolution kernels on the whole real line. Therefore, these results cannot apply to the downstream interaction kernels $\omega_i$ considered in~\eqref{multiChiarello} and~\eqref{multiclasse}, which necessarily have a jump discontinuity in $0$.

To account for drivers' reaction times,~\cite{CiaramagliaGoatinPuppo2024,KeimerPflug2019} proposed scalar models which incorporate the same type of non-locality of~\eqref{multiChiarello} adding a positive reaction time $\tau$. In particular, the model proposed in~\cite{CiaramagliaGoatinPuppo2024} reads 
\begin{equation}\label{oneclass}
\partial_t \rho(t,x)+\partial_x\left(\rho(t,x)f(\rho(t,x))v((\rho\ast \omega)(t-\tau,x))\right)=0,
\end{equation}
where the saturation term $f$ guarantees that a maximum principle holds~\cite{ChiarelloGoatin2018,CiaramagliaGoatinPuppo2024}, i.e. $\rho(t,x)\in [0,R]$ for all $(t,x)\in\R^+\times\R$, ensuring global existence. Moreover, 
solutions depend continuously 
on the time delay parameter~\cite[Corollary 4.4]{CiaramagliaGoatinPuppo2024}, implying that for $\tau\searrow 0$, the solution of \eqref{oneclass} converges in the $\L1$ norm to the unique entropy solution of the non-local model with no delay~\cite{ChiarelloGoatin2018}
 $$
 \partial_t \rho(t,x)+\partial_x\left(\rho(t,x)f(\rho(t,x))v((\rho\ast \omega)(t,x))\right)=0.
 $$
 We emphasize that, typically, delay introduces instability in the model increasing solutions' total variation, whereas larger visibility distance dampens oscillations.
 
In this paper, we extend \eqref{oneclass} to the multi-class case~\eqref{multiChiarello} by incorporating delays and allowing for class-specific maximal densities~\cite{FanWork2015}. 
It is important to notice that, since we aim to study the interaction between the delayed class of HVs and the instantaneously reactive AVs, we allow for zero delays in our model. 

\subsection{Main results and paper organization}

In this work, we prove global in time well-posedness results for the non-local system with delay~\eqref{multiclasse} with initial data~\eqref{eq:initial_datum} of bounded variation.
The proofs are based on uniform compactness estimates for a sequence of approximate solutions and on a Kru\v{z}kov-type doubling of variable technique.

The paper is organized as follows. In Section~\ref{sec:FVapprox} we present the finite volume scheme used to construct a sequence of approximate solutions to~\eqref{multiclasse} and in Section~\ref{multiconvergencemono} we provide the $\L\infty$ and $\BV$ estimates ensuring its convergence.
Section~\ref{sec:convergence} is devoted to the main results of the paper, that are given in Theorem~\ref{multiE1} and Theorem~\ref{stabilityteo}: In Theorem~\ref{multiE1} we prove the existence of the solution together with a maximum principle, the bound on the total variation and the $\L1$ stability in time; 
Theorem~\ref{stabilityteo} is an $\L1$ stability result of the solution with respect to the initial data and the delays, from which follows uniqueness of solutions and their convergence  to the solutions of the associated non-delayed model as the delay tends to zero. This also provides a global existence result for the non-delayed model, improving the results in~\cite{ChiarelloGoatin2019}, where the absence  of saturation functions did not allow to obtain global $\L\infty$ estimates. \\
Finally, in Section~\ref{sec:numerics}, we collect the  numerical studies of the model properties, including  investigations on the impact of saturation terms and delays. We particularly focus on the stabilizing effect induced by the presence of AVs in a mixed autonomous/human driven environment. 
Some technical details about the $\BV$ estimates and the proof of Theorem~\ref{multiE1}  are reported in Appendix~\ref{sec: app2} and Appendix~\ref{sec:app3}, respectively.

\section{Finite volume approximation}
\label{sec:FVapprox}

Throughout the paper, we require the following hypotheses.
\begin{assumption}\label{multi hp}
    For each $i=1,\dots,M$, it holds:
    \begin{itemize}
        \item $v_i\in\C2([0,+\infty);[0,V_i])$, $v_i'\leq 0$, $v_i(0)=V_i$ and $v_i(r)=0$ for $r\geq R_i$;
        \item $f_i\in\C1([0,R_i];[0,1])$, $f_i'\leq 0$, $f_i(0)=1$ and $f_i(R_i)=0$;
        \item $\omega_i\in\C1([0,L_i],\R^+)$, $\omega_i'(x)\leq 0$ and  
        \begin{equation}\label{Ji}
        \int_0^{L_i}\omega_i(s)\d s=J_i>0.
        \end{equation}
    \end{itemize}
    We extend the kernels $\omega_i(x)=0$ for $x>L_i$, the saturation functions  $f_i(\rho)=0$ for $\rho>R_i$, and $f_i(\rho)=1$ for $\rho<0$. 
\end{assumption}




We take a space step $\Delta x$ and we fix the kernel supports $L_i>0$ such that $L_i=N_i\Delta x$ for some $N_1,\dots,N_M\in\N$ holds. Moreover, we assume for simplicity that there exists a time step $\Delta t$ such that all the time delay parameters satisfy $\tau_i=h_i\Delta t$ for some $h_i\in\mathbb{N}_0:=\N\cup\{0\}$. 
We discretize \eqref{multiclasse} on a fixed grid made up of the cell centers $x_j=(j-\frac{1}{2})\Delta x$, the cell interfaces $x_{j+\frac{1}{2}}=j\Delta x$ for $j\in\Z$, and the time mesh $t^n=n\Delta t$.
We want to build a finite volume approximate solution, denoted as $\boldsymbol\rho^{\Delta x}(t,x)=(\rho^{\Delta x}_1,\dots,\rho^{\Delta x}_M)$, where $\rho^{\Delta x}_i=\rho^n_{i,j}$ for $(t,x)\in[t_n,t_{n+1})\times [x_{j-\frac{1}{2}},x_{j+\frac{1}{2}})$.
In order to do this, first we approximate each component $i$ of the initial data $\boldsymbol\rho^0$ with the piecewise constant function
$$
\rho^0_{i,j}=\frac{1}{\Delta x}\int_{x_{j-\frac{1}{2}}}^{x_{j+\frac{1}{2}}} \rho_i^0(x)\d x, \qquad j\in\mathbb{Z},
$$
 and we set $\rho^{-n}_{i,j}=\rho^0_{i,j}$ for $j\in\mathbb{Z}$,  $n=0,\dots,\max_lh_l$, consistently with \eqref{eq:initial_datum}.
Similarly, for the kernel, we define
\begin{equation}\label{wi}
\omega_i^k:=\frac{1}{\Delta x}\int_{k\Delta x}^{(k+1)\Delta x}\omega_i(x)\d x, \qquad k\in\mathbb{N}_0,
\end{equation}
so that from \eqref{Ji} we get $\Delta x\sum_{k=0}^{+\infty}\omega_i^k=\Delta x\sum_{k=0}^{N_i-1}\omega_i^k=J_i$, where the sum is finite since $\omega_i^k=0$ for $k\geq N_i$ sufficiently large.
Then,  we approximate the convolution term by a composite quadrature formula, denoting 
\begin{equation}\label{Vnj}
V^n_{i,j}:=v_i\left( \Delta x\sum_{k=0}^{+\infty}\omega_i^kr^n_{j+k}\right)=v_i\left( \Delta x\sum_{k=0}^{N_i-1}\omega_i^kr^n_{j+k}\right),
\end{equation}
where we indicate $r^n_j=\rho^n_{1,j}+\dots,\rho^n_{M,j}$.
Thus, following \cite{ChiarelloGoatin2019,2018Gottlich}, we consider the following Hilliges-Weidlich (HW) numerical flux \cite{Hilliges1995}
$$
\mathcal{F}^n_{i,j+\frac{1}{2}}=\rho^n_{i,j}f_i(\rho^n_{i,j+1})V^{n-h_i}_{i,j+1},
$$
with $\lambda=\Delta t/\Delta x$, which leads to the following numerical scheme
\begin{equation}\label{multischema}
\rho^{n+1}_{i,j}=\rho^n_{i,j}-\lambda\left(\rho^n_{i,j}f_i(\rho^n_{i,j+1})V^{n-h_i}_{i,j+1}-\rho^n_{i,j-1}f_i(\rho^n_{i,j})V^{n-h_i}_{i,j}\right).
\end{equation}

To prove the convergence of the numerical scheme, we start providing some important properties of~\eqref{multischema}, which support the reliability of the approximate solution as a density~\cite{ACG2015}. Particularly, we prove that each component of the solution remains always bounded between $0$ and the relative maximum capacity of the road $R_i$. 

\begin{lemma}[Positivity]\label{multiposteo} 
Under the CFL condition
\begin{equation} \label{eq:multiCFL1}
\lambda\leq\frac{1}{\max_iV_i},
\end{equation}
then the scheme \eqref{multischema} is positivity preserving on $[0,T]\times\R $ for all $T>0$.
\end{lemma}
\begin{proof}
Let us assume that $\rho^{n-l}_{i,j}\geq 0$ for all $j\in\mathbb{Z}$, $i=1,\dots,M$ and $l=0,\dots,\max_ih_i$. Then, we have
$$
\rho^{n+1}_{i,j}=\left(1-\lambda f_i(\rho^n_{i,j+1})V^{n-h_i}_{i,j+1}\right)\rho^n_{i,j}+\lambda f_i(\rho^n_{i,j})V^{n-h_i}_{i,j}\rho^n_{i,j-1}\geq 0,
$$
since all the coefficients are non-negative.
\end{proof}

\begin{lemma}[$\L 1$-bound]\label{multiL1boundteo}
For any $n\in\mathbb{N}$, under the CFL condition \eqref{eq:multiCFL1} the approximate solutions constructed by scheme \eqref{multischema} satisfy
$$
\norma{\rho^n_i}_1=\norma{\rho_i^0}_1,
$$
where $\norma{\rho^n_i}_1:=\Delta x\sum_j\modulo{\rho^n_{i,j}}$ denotes the $\L 1$-norm of $\rho_i^{\Delta x}(n\Delta t,\cdot)$.
\end{lemma}

\begin{proof}
Thanks to Lemma \ref{multiposteo}, we have
\begin{align*}
\|\rho_i^{n+1}\|_1=\ \Delta x\sum_j\rho^{n+1}_{i,j}&=\Delta x\sum_j\rho^n_{i,j}-\lambda\Delta x\sum_j\left(\mathcal{F}^n_{i,j+\frac{1}{2}}-\mathcal{F}^n_{i,j-\frac{1}{2}}\right) \\
&=\Delta x\sum_j\rho^n_{i,j}-\lambda\Delta x\left(\sum_j\mathcal{F}^n_{i,j+\frac{1}{2}}-\sum_j\mathcal{F}^n_{i,j+\frac{1}{2}}\right) \\
&=\Delta x\sum_j\rho^n_{i,j}=\norma{\rho_i^n}_1.
\end{align*}
\end{proof}

\begin{lemma}[$\L\infty$-bound / weak maximum principle]\label{multiboundteo}
If $\rho^0_{i,j}\in[0,R_i]$ for all $j\in\mathbb{Z}$, $i=1,\dots,M$, and the CFL condition
\begin{equation}\label{multiCFL}
\lambda\leq\frac{1}{\max_i\left\{V_i\left(1+R_i\norma{f_i'}\right)+\Delta xR_i\norma{\omega_i}\norma{v_i'}\right\}}
\end{equation}
holds, then $\rho_{i,j}^n \in [0,R_i]$ for all $j\in\Z$, $i=1,\dots,M$ and $
n\in\N$. 
\end{lemma}

\begin{proof}
We can rewrite the scheme \eqref{multischema} in the form $\rho^{n+1}_{i,j}=\Phi(\rho^n_{i,j})$, being 
$$
\Phi(\rho^n_{i,j}):=\rho^n_{i,j}-\lambda\left(\rho^n_{i,j}f_i(\rho^n_{i,j+1})V^{n-h_i}_{i,j+1}-\rho^n_{i,j-1}f_i(\rho^n_{i,j})V^{n-h_i}_{i,j}\right).
$$
 We observe that under the CFL condition it holds
 \begin{align*}
     \Phi(0)=&\ \lambda\rho^n_{i,j-1}V^{n-h_i}_{i,j}\leq\lambda R_i V_i\leq R_i,\\
     \Phi(R_i)=&\ R_i-\lambda R_if_i(\rho^n_{i,j+1})V^{n-h_i}_{i,j+1}\leq R_i.
 \end{align*}
To conclude the proof, it is enough to show that $\Phi$ is monotone with respect to $\rho^n_{i,j}$. This follows again from the CFL condition since for $h_i\geq 1$
$$
\frac{\partial\Phi}{\partial\rho^n_{i,j}}(\rho^n_{j})=1-\lambda\left(f_i(\rho^n_{i,j+1})V^{n-h_i}_{i,j+1}-\rho^n_{i,j-1}f_i'(\rho^n_{i,j})V^{n-h_i}_{i,j}\right)\geq 0,
$$ 
and also in the non-delayed case $h_i=0$ we have 
\begin{align*}
    \frac{\partial\Phi}{\partial\rho^n_{i,j}}(\rho^n_{j})=&\ 1-\lambda\left(f_i(\rho^n_{i,j+1})V^n_{i,j+1}-\rho^n_{i,j-1}f_i'(\rho^n_{i,j})V^n_{i,j}\right)\\
    &+\lambda\Delta x\omega^0_i\rho^n_{i,j-1}f_i(\rho^n_{i,j})v_i'\left(\Delta x\sum_{k=0}^{N_i-1}\omega^k_ir^n_{j+k}\right)\geq 0.
\end{align*}
\end{proof}


\begin{remark} \label{rem:invariant_domain}
{\rm
The saturation function was introduced to indicate the free space available on the road, thus modeling the influence of road capacity on drivers, who slow down in particularly congested traffic conditions. This behavior does not depend on the type of vehicles congesting the road. Therefore, it would be reasonable to assume that each saturation function $f_i$ depends on the total density $r$, rather than on the individual density $\rho_i$. This would lead to the following modified model
\begin{equation}\label{eq:biclasse_total}
\partial_t \rho_i(t,x)+\partial_x\left(\rho_i(t,x)f_i(r(t,x))v_i((r\ast \omega_i)(t-\tau_i,x))\right)=0,\qquad\forall i=1,\dots,M,
\end{equation}
where
$$
f_i\in\C1([0,R];[0,1])~\mbox{ s.t. } ~f_i'\leq 0,~ f_i(0)=1~\mbox{ and }~f_i(R)=0,
$$
assuming $R_1=\dots=R_M=:R$.
Indeed, for the modified model \eqref{eq:biclasse_total}, the simplex
$$
\mathcal{S}:=\left\{\boldsymbol\rho\in\R^M~\left|~\sum_{i=1}^M\rho_i\leq R,~\rho_i\geq 0~\mbox{ for }i=1,\dots,M\right.\right\}
$$
is an invariant domain, as it holds for the classical (local) multi-population model~\cite{BenzoniColombo2003}, see~\cite[Lemma 2.3]{ChiarelloGoatin2019}. 

\begin{lemma}[$\L\infty$-bound / weak maximum principle]\label{multiboundteo_mod}
Under the CFL condition~\eqref{multiCFL}, which now reads 
\begin{equation} \label{CFLsimplified}
\lambda\leq\frac{1}{\max_i\left\{V_i\left(1+R\norma{f_i'}\right)+\Delta xR\norma{\omega_i}\norma{v_i'}\right\}},
\end{equation}
for any initial datum $\boldsymbol\rho^0_j=(\rho^0_{1,j},\dots,\rho^0_{M,j})\in\mathcal{S}$ for all $j\in\Z$, the approximate solutions to \eqref{eq:biclasse_total} computed by the scheme 
\begin{equation}
\rho^{n+1}_{i,j}=\rho^n_{i,j}-\lambda\left(\rho^n_{i,j}f_i(r^n_{j+1})V^{n-h_i}_{i,j+1}-\rho^n_{i,j-1}f_i(r^n_{j})V^{n-h_i}_{i,j}\right)\label{scheme_mod}
\end{equation}
satisfy the following uniform bounds:
$$
\boldsymbol\rho^n_j=(\rho^n_{1,j},\dots,\rho^n_{M,j})\in\mathcal{S},\qquad\forall j\in\Z,~n\in \N.
$$
\end{lemma}

\begin{proof}
The positivity of the scheme \eqref{scheme_mod} can be shown easily as in Lemma \ref{multiposteo}.
\\To prove that $r_j^n\leq R$ for all $j\in\Z,~n\in \N$, we follow closely the proof of Lemma~\ref{multiboundteo}, see also~\cite[Lemma 2.3]{ChiarelloGoatin2019}. We assume $\boldsymbol\rho^n_j\in\mathcal{S}$ for all $j\in\Z$. Summing~\eqref{scheme_mod} over $i=1,\dots,M$ we get
$$
        r^{n+1}_j=r^n_{j}-\lambda\left(\sum_{i=1}^M\rho^n_{i,j}f_i(r^n_{j+1})V^{n-h_i}_{i,j+1}-\sum_{i=1}^M\rho^n_{i,j-1}f_i(r^n_j)V^{n-h_i}_{i,j}\right)=:\ \Phi(\boldsymbol\rho^n_j).
$$
We have
\begin{align*}
    \Phi(0,\dots,0)= \lambda\sum_{i=1}^M\rho^n_{i,j-1}V^{n-h_i}_{i,j}\leq\lambda V_iR\leq R \quad \hbox{ by~\eqref{CFLsimplified}},
\end{align*}
and, for $\boldsymbol\rho^n_j\in\mathcal{S}$ such that $r^n_j=R$,
\begin{align*}
    \Phi(\boldsymbol\rho^n_j)= R-\lambda\sum_{i=1}^M\rho^n_{i,j}f_i(r^n_{j+1})V^{n-h_i}_{i,j+1}\leq R.
\end{align*}
As in the proof of Lemma~\ref{multiboundteo}, 
the claim follows from the monotonicity of $\Phi$ with respect to each variable, 
as $\frac{\partial\Phi}{\partial\rho^n_{i,j}}\geq 0$ for all $i=1,\dots,M$, by~\eqref{CFLsimplified}.
\end{proof}

However, system~\eqref{eq:biclasse_total} consists of equations coupled also in the local components $f_i$, $i=1,\ldots,M$, and not only in the non-local terms.
For these systems, $\BV$ estimates are not available in general and this is the main reason we are considering both cases in this work.
A numerical comparison of~\eqref{multiclasse} and~\eqref{eq:biclasse_total} is presented in Section \ref{saturationtotal}.
}\end{remark}

\section{$\BV$ estimates}\label{multiconvergencemono}
The proofs in this section follow closely~\cite{CiaramagliaGoatinPuppo2024}, which we refer to for further details on the sketched calculations. 

We start with the $\BV$ estimate in space, noticing that we always intend $\BV\subset\L1$.

\begin{proposition}[Spatial $\BV$-bound]\label{multispaceBVteo}
    Let Assumption~\ref{multi hp} and the CFL condition \eqref{multiCFL} hold. Then, for any initial data $\boldsymbol\rho^0\in\BV(\R;[0,R_1]\times\dots\times[0,R_M])$, the numerical solution $\boldsymbol\rho^{\Delta x}$ given by the HW scheme \eqref{multischema} has bounded total variation for $t\in[0,T]$, uniformly in $\Delta x$, for every time horizon $T>0$.
\end{proposition}

\begin{proof}
    Let us set
    $
    \Delta^n_{i,j+\frac{1}{2}}:=\rho^n_{i,j+1}-\rho^n_{i,j}.
    $
Using the mean value theorem, we obtain
\begin{align}
\Delta^{n+1}_{i,j+\frac{1}{2}}=\ &\ \left[1-\lambda\left(f_i(\rho^n_{i,j+2})V^{n-h_i}_{i,j+2}-\rho^n_{i,j-1}f_i'(\tilde{\rho}^n_{i,j+\frac{1}{2}})V^{n-h_i}_{i,j+1}\right)\right]\Delta ^n_{i,j+\frac{1}{2}}\label{1}\\
&-\lambda\rho^n_{i,j}f_i'(\tilde{\rho}^n_{i,j+\frac{3}{2}})V^{n-h_i}_{i,j+2}\Delta^n_{i,j+\frac{3}{2}}+\lambda f_i(\rho^n_{i,j+1})V^{n-h_i}_{i,j+1}\Delta^n_{i,j-\frac{1}{2}}\label{2}\\
&-\lambda\left(f_i(\rho^n_{i,j+1})\Delta^n_{i,j-\frac{1}{2}}+\rho^n_{i,j-1}f_i'(\tilde{\rho}^n_{i,j+\frac{1}{2}})\Delta^n_{i,j+\frac{1}{2}}\right)\left(V^{n-h_i}_{i,j+2}-V^{n-h_i}_{i,j+1}\right)\label{3}\\
&-\lambda\rho^n_{i,j-1}f_i(\rho^n_{i,j})\underbrace{\left[\left(V^{n-h_i}_{i,j+2}-V^{n-h_i}_{i,j+1}\right)-\left(V^{n-h_i}_{i,j+1}-V^{n-h_i}_{i,j}\right)\right]}_{(\ast)}\label{4},
\end{align}
with $\tilde{\rho}^n_{i,j+\frac{1}{2}}$ between $\rho^n_{i,j}$ and $\rho^n_{i,j+1}$ for all $j\in\mathbb{Z}$.
The term $(\ast)$ in \eqref{4} can be estimated as
    \begin{eqnarray*}
        (\ast)&=&v_i''(\tilde{\xi}_{i,j+1})\left[\xi_{i,j+\frac{3}{2}}-\xi_{i,j+\frac{1}{2}}\right]\Delta x\sum_{k=1}^{+\infty}\omega_i^{k-1}\sum_{l=1}^M\Delta^{n-h_i}_{l,j+k+\frac{1}{2}}\\
        &&+v_i'(\xi_{i,j+\frac{1}{2}})\Delta x\left(\sum_{k=1}^{+\infty}(\omega_i^{k-1}-\omega_i^k)\sum_{l=1}^M\Delta^{n-h_i}_{l,j+k+\frac{1}{2}}-\omega_i^0\sum_{l=1}^M\Delta^{n-h_i}_{l,j+\frac{1}{2}}\right),
    \end{eqnarray*}
    for some $\xi_{i,j+\frac{3}{2}}$ between $\Delta x\sum_{k=0}^{+\infty}\omega_i^kr^{n-h_i}_{j+k+2}$ and $\Delta x\sum_{k=0}^{+\infty}\omega_i^kr^{n-h_i}_{j+k+1}$, and $\xi_{i,j+\frac{1}{2}}$ between $\Delta x\sum_{k=0}^{+\infty}\omega_i^kr^{n-h_i}_{j+k+1}$ and $\Delta x\sum_{k=0}^{+\infty}\omega_i^kr^{n-h_i}_{j+k}$, and some $\tilde{\xi}_{i,j+1}$ between $\xi_{i,j+\frac{1}{2}}$ and $\xi_{i,j+\frac{3}{2}}$. 
    Since from Lemma \ref{multiboundteo} and the monotonicity of $\omega_i$ we get
    \begin{eqnarray}
        \modulo{\xi_{i,j+\frac{3}{2}}-\xi_{i,j+\frac{1}{2}}}&\leq&\Delta x\left[\sum_{k=1}^{+\infty}\big[\theta\omega_i^{k-1}+(1-\theta)\omega_i^k-\mu\omega_i^k-(1-\mu)\omega_i^{k+1}\big]+4\omega_i^0\right]R\nonumber\\
        &\leq&\Delta x\left[\sum_{k=1}^{+\infty}[\omega_i^{k-1}-\omega_i^{k+1}]+4\omega_i^0\right]\sum_{l=1}^MR_l\leq6\Delta x \omega_i^0R,\label{multidelta xi}
    \end{eqnarray}
with $R:=\sum_{l=1}^MR_l$, then we obtain the bound 
\begin{equation}
\sum_j\modulo{\eqref{4}}\leq\Delta t\mathcal{H}_i\sum_{l=1}^M\sum_j\modulo{\Delta^{n-h_i}_{l,j+\frac{1}{2}}},\label{bound4}
\end{equation}
for $\mathcal{H}_i=R_i\norma{\omega_i}\left(6\norma{v_i''}J_iR+2\norma{v_i'}\right)$.   
Regarding the term \eqref{3}, the bound
\begin{eqnarray}
        \modulo{V^{n-h_i}_{i,j+1}-V^{n-h_i}_{i,j}}&\leq&\norma{v_i'}\Delta x\modulo{\sum_{k=1}^{+\infty}(\omega_i^{k-1}-\omega_i^k)r^{n-h_i}_{j+k}-\omega_i^0r^{n-h_i}_j}\label{multistimamodulo_parziale}\\
        &\leq&2\Delta x\norma{v_i'}\norma{\omega_i}R,\label{multistimamodulo}
\end{eqnarray}
that holds for all $j\in\Z$, implies
\begin{align}
\sum_j\modulo{\eqref{3}}\leq&\ \lambda\sum_j\modulo{V^{n-h_i}_{i,j+3}-V^{n-h_i}_{i,j+2}}\modulo{\Delta^n_{i,j+\frac{1}{2}}}+\lambda R_i\norma{f_i'}\sum_j\modulo{V^{n-h_i}_{i,j+2}-V^{n-h_i}_{i,j+1}}\modulo{\Delta^n_{i,j+\frac{1}{2}}}\nonumber\\
\leq&\ \Delta t\mathcal{G}_i\sum_j\modulo{\Delta^n_{i,j+\frac{1}{2}}},\label{bound3}
\end{align}
being $\mathcal{G}_i=2\norma{v_i'}\norma{\omega_i}\left(1+R_i\norma{f_i'}\right)R$. 
Thus, taking the absolute values in the bound of $\Delta^{n+1}_{i,j+\frac{1}{2}}$ at the beginning of the proof, summing on $j$ and using the CFL condition ensures
$$
\sum_j\modulo{\Delta^{n+1}_{i,j+\frac{1}{2}}}\leq(1+\Delta t\mathcal{G}_i)\sum_j\modulo{\Delta^n_{i,j+\frac{1}{2}}}+\Delta t\mathcal{H}_i\sum_{l=1}^M\sum_j\modulo{\Delta^{n-h_i}_{l,j+\frac{1}{2}}}.
$$
Moreover, if we sum over $i=1,\dots,M$ and we set $\tv^n:=\sum_{l=1}^M\sum_j\modulo{\Delta^n_{l,j+\frac{1}{2}}}$, this leads to
\begin{equation}
\tv^{n+1}\leq\left(1+\Delta t\mathcal{G}\right)\tv^n+\Delta t\mathcal{H}\sum_{i=1}^M\tv^{n-h_i},\label{sequencetemp}
\end{equation}
where $\mathcal{H}=\max_i\mathcal{H}_i$ and $\mathcal{G}=\max_i\mathcal{G}_i$.\\
We distinguish between the cases:
\begin{itemize}
    \item the model has no delay, meaning $h_i=0$ for every $i=1,\dots,M$;
    \item at least one of the time delay parameters is strictly positive.
\end{itemize}
In the non-delayed scenario, \eqref{sequencetemp} becomes 
$$
\tv^{n+1}\leq\left(1+\Delta t(\mathcal{G}+M\mathcal{H})\right)\tv^n,
$$
which can be iterated getting
$$
\tv^n\leq\left(1+\Delta t(\mathcal{G}+M\mathcal{H})\right)^n\tv^0.
$$
Then, passing to the limit as $\Delta t\rightarrow 0$, we obtain
$$
\sum_{i=1}^M\tv\left(\rho_i^{\Delta x}(T,\cdot)\right)\leq e^{(\mathcal{G}+M\mathcal{H})T}\sum_{i=1}^M\tv(\rho_i^0).
$$
Now we consider the properly delayed case and we assume for simplicity that $M=2$. The result holds also in the general case of $M$ classes (see Remark \ref{casogen}). 
We denote
$$
h_{\min}=\min\{h_1,h_2\}\qquad\mbox{ and }\qquad h_{\max}=\max\{h_1,h_2\}
$$
and we set $\tau_{\min}=h_{\min}\Delta t$ and $\tau_{\max}=h_{\max}\Delta t$. We distinguish again between two cases:
\begin{itemize}
    \item[(i)] the sole non-zero delay is $\tau_{\max}$ ($h_{\max}\geq 1$ and $h_{\min}=0$);
    \item[(ii)] both $h_{\min}$ and $h_{\max}$ are non-zero.
\end{itemize}
In case (i), \eqref{sequencetemp} becomes
$$
\tv^{n+1}\leq\left(1+\Delta t(\mathcal{G}+\mathcal{H})\right)\tv^n+\Delta t\mathcal{H}\tv^{n-h_{\max}},
$$
where now $\tv^n=\sum_j\modulo{\Delta^n_{1,j+\frac{1}{2}}}+\sum_j\modulo{\Delta^n_{2,j+\frac{1}{2}}}$.
Recalling that \eqref{eq:initial_datum} ensures $\tv^{-l}=\tv^0$ for all $l$, as in \cite{CiaramagliaGoatinPuppo2024} we obtain the estimate
\begin{equation}\label{partial_bvbound}
\tv\left(\rho_1^{\Delta x}(T,\cdot)\right)+\tv\left(\rho_2^{\Delta x}(T,\cdot)\right)\leq\left(2e^{(\mathcal{G}+\mathcal{H})t}-1\right)\left(2e^{(\mathcal{G}+\mathcal{H})\tau_{\max}}-1\right)^{\big\lfloor\frac{T}{\tau_{\max}}\big\rfloor}\left(\tv(\rho^0_1)+\tv(\rho^0_2)\right),
\end{equation}
where we set $T=t+\lfloor T/\tau_{\max}\rfloor\tau_{\max}$.\\
Otherwise, in case (ii) \eqref{sequencetemp} translates into
\begin{equation}
\tv^{n+1}\leq(1+\Delta t\mathcal{G})\tv^n+\Delta t\mathcal{H}\left(\tv^{n-h_{\max}}+\tv^{n-h_{\min}}\right).
\label{sequence}
\end{equation}
In such a case, for the first $h_{\min}$ terms of the sequence \eqref{sequence} we get as above
\begin{align*}
    \tv^1&\leq(1+\Delta t\mathcal{G})\tv^0+2\Delta t\mathcal{H}\tv^0,\\
    &~~\vdots\\
    \tv^{h_{\min}}&\leq(1+\Delta t\mathcal{G})\tv^{h_{\min}-1} +2\Delta t\mathcal{H}\tv^0\\
    &\leq(1+\Delta t\mathcal{G})^{h_{\min}}\tv^0 +2\Delta t\mathcal{H}\tv^0\sum_{k=0}^{h_{\min}-1}(1+\Delta t\mathcal{G})^k\\
    &\leq (1+\Delta t\mathcal{G})^{h_{\min}}\tv^0+2\left((1+\Delta t\mathcal{M})^{h_{\min}}-1\right) \tv^0,
\end{align*}
with $\mathcal{M}:=\max\{\mathcal{H},\mathcal{G}\}$, and this implies the bound
\begin{equation}\label{boundtemp}
\tv^k\leq\left(3(1+\Delta t\mathcal{M})^{h_{\min}}-2\right)\tv^0=:\mathcal{B}_{\min}^{\Delta t}\tv^0,\qquad k=0,\dots,h_{\min}.
\end{equation}
Regarding the following terms, for each $k=1,\dots,\lfloor h_{\max}/h_{\min}\rfloor$ we can write
$$
\tv^{n-h_{\max}}=\tv^0,\quad\tv^{n-h_{\min}}\leq(\mathcal{B}_{\min}^{\Delta t})^{k-1}\tv^0,\quad\mbox{ for all }n=(k-1)h_{\min}+1,\dots,kh_{\min},
$$
and this means that we can reiterate the same argument getting
$$
\tv^k\leq\left(\mathcal{B}_{\min}^{\Delta t}\right)^{\big\lfloor\frac{h_{\max}}{h_{\min}}\big\rfloor}\tv^0,\qquad k=0,\dots,\Big\lfloor\frac{h_{\max}}{h_{\min}}\Big\rfloor h_{\min},
$$
and also 
\begin{equation}
\tv^k\leq\left(3(1+\Delta t\mathcal{M})^{h_{\max}-\big\lfloor\frac{h_{\max}}{h_{\min}}\big\rfloor h_{\min}}-2\right)\left(\mathcal{B}_{\min}^{\Delta t}\right)^{\big\lfloor\frac{h_{\max}}{h_{\min}}\big\rfloor}\tv^0,\qquad k=0,\dots,h_{\max}.\label{boundtemp2}
\end{equation}
Similarly, observing that $h_i=\tau_i/\Delta t$ for $i=1,2$, and passing to the limit as $\Delta t\rightarrow 0$, then for  $T=t+\lfloor T/\tau_{\max}\rfloor\tau_{\max}$ we obtain
\begin{align}
\tv\left(\rho_1^{\Delta x}(T,\cdot)\right)&+\tv\left(\rho_2^{\Delta x}(T,\cdot)\right)\leq\left(3e^{\mathcal{M}t}-2\right)\left(3e^{\mathcal{M}\tau_{\max}}-2\right)^{\big\lfloor\frac{T}{\tau_{\max}}\big\rfloor}\label{TVestimate}\\
&\cdot\left(3e^{\mathcal{M}\left(\tau_{\max}-\big\lfloor\frac{\tau_{\max}}{\tau_{\min}}\big\rfloor\tau_{\min}\right)}-2\right)\left(3e^{\mathcal{M}\tau_{\min}}-2\right)^{\big\lfloor\frac{\tau_{\max}}{\tau_{\min}}\big\rfloor}
\left(\tv(\rho_1^0)+\tv(\rho_2^0)\right),\nonumber
\end{align}
which leads to the statement.
\end{proof}

\begin{remark}[Dependence on the parameters]\label{relazioneoss}
{\rm
The estimates provided in the last part of the proof show the dependence of the total variation on the delays and the look-ahead distances. Indeed, the positive constants $\mathcal{G}$ and $\mathcal{H}$ are linked to the norm $\norma{\omega_i}$. Since if $\omega_i\in\C{1}([0,1];\R ^+)$ then we can choose a re-scaled kernel function such as
    $$
    \omega_{i,L_i}(x)=\frac{1}{L_i}\omega_i\left(\frac{x}{L_i}\right),
    $$
this means that $\mathcal{H},\mathcal{G}\sim 1/L_i$ for all $i=1,2$. As a consequence, the positive constant
 \begin{align}
    \mathcal{M}&:=2\max\left\{\max_{1\leq i\leq M}R_i\norma{\omega_i}\left(3\norma{v_i''}J_iR+\norma{v_i'}\right)\right.,\nonumber\\
    &\left.\sum_{l=1}^MR_l\max_{1\leq i\leq M}\norma{v_i'}\norma{\omega_i}\left(1+R_i\norma{f_i'}\right)+(M-\modulo{\mathcal{J}})\max_{1\leq j \leq M}R_j\norma{\omega_j}\left(3\norma{v_j''}J_jR+\norma{v_j'}\right)\right\},\label{eq:M}
\end{align}
is dimensionally the inverse of time and it is also $\mathcal{M}\sim 1/L_i$. Thus, the non-locality in space gives stability to the solution.
\\Regarding the dependence on the delay, the bounds \eqref{partial_bvbound} and \eqref{TVestimate} imply that increasing delays leads to higher values for the total variation bounds of the solution.
}\end{remark} 
    
\begin{remark}[General case of $M$ classes]\label{casogen} 
{\rm
In the proof we assumed for simplicity $M=2$. 
Now we generalize the estimates for the $M$ classes model, leaving some details in Appendix \ref{sec: app2}. Using the convention for which doing the product on the empty set gives one, as in the proof we can obtain the general bound
    \begin{equation}
    \sum_{i=1}^M\tv\left(\rho_i^{\Delta x}(T,\cdot)\right)\leq\mathcal{B}_T\prod_{j\in\tilde{\mathcal{J}}\atop T\geq\tau_j}\mathcal{B}_j\sum_{i=1}^M\tv(\rho_i^0),\label{eq:TVestimate}
    \end{equation}
    being  the set containing the classes with non-zero delays
    \begin{equation}
    \mathcal{J}:=\left\{i=1,\dots,M\ |\ \tau_i>0\right\}\label{J}
    \end{equation}
    and $\tilde{\mathcal{J}}:=\mathcal{J}/\sim$ the subset defined in \eqref{Jtilde} where identical delay values are counted only once.
   The positive constants $\mathcal{B}_j,\mathcal{B}_T$ are defined respectively in \eqref{Bj} and \eqref{BT} and depend on $\boldsymbol L=\left(L_1,\dots,L_M\right)$ and $\boldsymbol{\tau}$. As an example, we consider the case of $M$ classes having all the same delay $\tau>0$. Then, the constants in \eqref{eq:TVestimate} become
   $$
   \mathcal{B}_T\prod_{j\in\tilde{\mathcal{J}}\atop T\geq\tau_j}\mathcal{B}_j=\left((M+1)e^{\mathcal{M}(T-\lfloor T/\tau\rfloor\tau)}-M\right)\left((M+1)e^{\mathcal{M}\tau}-M\right)^{\lfloor T/\tau\rfloor-1}.
   $$
}\end{remark}

\smallskip

The space-time total variation estimate derives from the following result.

\begin{lemma}[$\L1$ Lipschitz continuity in time]\label{multiL1contteo}
Let Assumption \ref{multi hp} and the CFL condition \eqref{multiCFL} hold. Then, for any initial condition $\boldsymbol\rho^0\in\BV(\R;[0,R_1]\times\dots[0,R_M])$, the approximate solution constructed via the scheme~\eqref{multischema} satisfies 
$$
\norma{\rho_i^{\Delta x}(T,\cdot)-\rho_i^{\Delta x}(T-t,\cdot)}_{1}\leq\mathcal{K}_it,\qquad\forall i=1,\dots,M,
$$
for any $T>0$ and $t\in[0,T]$, with $\mathcal{K}_i$ given by \eqref{k}.
\end{lemma}

\begin{proof}
The proof follows closely \cite[Remark 3.5]{CiaramagliaGoatinPuppo2024}. Let $N_T\in\mathbb{N}$ be such that $N_T\Delta t<T\leq(N_T+1)\Delta t$.
We recall from \eqref{multischema} and \eqref{multistimamodulo_parziale} that for every $j\in\mathbb{Z}$ and $n=0,\dots,N_T-1$
 \begin{align}
 \modulo{\rho^{n+1}_{i,j}-\rho^n_{i,j}}\leq&\ \lambda V_i\modulo{\rho^n_{i,j}-\rho^n_{i,j-1}}+\lambda R_i\norma{f_i'}V_i\modulo{\rho^n_{i,j+1}-\rho^n_{i,j}}\nonumber\\
 &+\lambda R_i\norma{v_i'}\Delta x\left(\omega^0_i\modulo{r^{n-h_i}_j}+\sum_{k=1}^{+\infty}(\omega^{k-1}_i-\omega^k_i)\modulo{r^{n-h_i}_{j+k}}\right)\label{multidelta rho}.
 \end{align}
Now, we fix $t=m\Delta t$, with $m\leq N_T$.
Using the last inequality,  Lemma \ref{multiL1boundteo} and Proposition \ref{multispaceBVteo}, we obtain 
\begin{align}
\sum_j\Delta x\modulo{\rho^{N_T}_{i,j}-\rho^{N_T-m}_{i,j}}\leq&\ t\left(1+R_i\norma{f_i'}\right)V_i
\sup_{s\in[0,T]}\tv\left(\rho_i^{\Delta x}(s,\cdot)\right)\nonumber\\
&+2tR_i \norma{v_i'}\omega_i^0\sup_{s\in[0,T]}\sum_{l=1}^M\norma{\rho_l^{\Delta x}(s,\cdot)}_1\leq\mathcal{K}_it,\label{multiBVstima}
\end{align}
with 
\begin{equation}\label{k}
\mathcal{K}_i:=\left(1+R_i\norma{f_i'}\right)V_i C(T,\boldsymbol L,\boldsymbol\tau)\sum_{l=1}^M\tv(\rho_l^0)+2R_i \norma{v_i'} \omega_i^0\sum_{l=1}^M\norma{\rho_l^0}_1,
\end{equation}
where the positive constant $C(T,\boldsymbol L,\boldsymbol\tau)$ is given by \eqref{eq:TVestimate} and it is 
\begin{equation}\label{c}
C(T,\boldsymbol L,\boldsymbol\tau):=\sup_{s\in[0,T]}\mathcal{B}_s\prod_{j\in\tilde{\mathcal{J}}\atop s\geq\tau_j}\mathcal{B}_j,
\end{equation}
with $\tilde{\mathcal{J}}$ defined in \eqref{Jtilde}, $\mathcal{B}_j$ in \eqref{Bj} and $\mathcal{B}_T$ in \eqref{BT}.
\end{proof}

We can now provide an estimate for the discrete total variation in space and time. 

\begin{proposition}[BV estimate in space and time]\label{multiBVteo}
Let Assumption \ref{multi hp} and the CFL condition \eqref{multiCFL} hold. Then, for any initial condition $\boldsymbol\rho^0\in\BV(\R;[0,R_1]\times\dots\times[0,R_M])$, the numerical solution $\boldsymbol\rho^{\Delta x}$ has uniformly bounded total variation on $[0,T]\times\R $, for any time horizon $T>0$.
\end{proposition}

\begin{proof}
    See proof of \cite[Proposition 3.6]{CiaramagliaGoatinPuppo2024}.
\end{proof}

\section{Well-posedness of entropy weak solutions}
\label{sec:convergence}

Before stating our main result,  we derive a discrete entropy inequality~\cite{ACG2015,ChiarelloGoatin2018, CiaramagliaGoatinPuppo2024} for the approximate solutions generated by \eqref{multischema}, which will be used to prove that the limit of the HW approximations is indeed an entropy solution in the sense of Definition \ref{multientropy}. \\
For every class $i=1,\dots,M$, let us denote 
\begin{align*}
    G_{i,j+\frac{1}{2}}(u,w)=&\ uf_i(w)V^{n-h_i}_{i,j+1}\\
    F^\kappa_{i,j+\frac{1}{2}}(u,w)=&\ G_{i,j+\frac{1}{2}}(u\wedge\kappa,w\wedge\kappa)-G_{i,j+\frac{1}{2}}(u\vee\kappa,w\vee\kappa), \\
    =&\ \sgn(u-\kappa)\left(uf_i(u)-\kappa f_i(\kappa)\right)V^{n-h_i}_{i,j}\\
&+\modulo{u-\kappa}\left[f_i(w\wedge\kappa)-f_i(u\wedge\kappa)\right]V^{n-h_i}_{i,j+1}\\
&+(u\vee\kappa)\left(\modulo{f_i(u)-f_i(\kappa)}-\modulo{f_i(w)-f_i(\kappa)}\right)V^{n-h_i}_{i,j+1}\\
&+\sgn(u-\kappa)\left(uf_i(u)-\kappa f_i(\kappa)\right)\left(V^{n-h_i}_{i,j+1}-V^{n-h_i}_{i,j}\right)
\end{align*}
with $a\wedge b=\max\{a,b\}$ and $a\vee b=\min\{a,b\}$.
Then the following property holds.

\begin{proposition}[Discrete entropy inequality]
Given Assumption \ref{multi hp}, let $\rho^n_{i,j}$, $j\in\mathbb{Z}$, $i=1,\dots,M$, $n\in\{-\max_lh_l,\dots,0\}\cup\mathbb{N}$, be given by the scheme \eqref{multischema}. Then, if the CFL condition \eqref{multiCFL} is satisfied, we have
\begin{align}
    \modulo{\rho^{n+1}_{i,j}-\kappa}-\modulo{\rho^n_{i,j}-\kappa}&+\lambda\left(F^\kappa_{i,j+\frac{1}{2}}(\rho^n_{i,j},\rho^n_{i,j+1})-F^\kappa_{i,j-\frac{1}{2}}(\rho^n_{i,j-1},\rho^n_{i,j})\right)\nonumber\\
    &+\lambda\sgn(\rho^{n+1}_{i,j}-\kappa)\kappa f_i(\kappa)\left(V^{n-h_i}_{i,j+1}-V^{n-h_i}_{i,j}\right)\leq 0,\label{multidisentropy}
\end{align}
for all $j\in\mathbb{Z}$, $i=1,\dots,M$, $n\in\mathbb{N}_0$, and $\kappa\in\R $.
\end{proposition}

The proof follows closely \cite[Proposition 3.4]{ChiarelloGoatin2018}.

Given the entropy inequality and the $\BV$ estimates given in Proposition~\ref{multiBVteo}, we are able to propose a proof of existence of solutions for every time horizon $T>0$. 

\begin{theorem}[Existence]\label{multiE1}
 Given Assumption \ref{multi hp}, for any $T>0$ the model \eqref{multiclasse}-\eqref{eq:initial_datum} admits an entropy weak solution $\boldsymbol\rho$ in the sense of Definition \ref{multientropy}, such that each component $i=1,\dots,M$ satisfies
 \begin{align}
   &  \rho_i(t,x) \in [0,R_i] &\hbox{for a.e. } x\in\R, t\in [0,T], \label{eq:multilinfty} \\
   &\norma{\rho_i(t,\cdot)}_1 = \|\rho_i^0\|_1 &\hbox{for } t\in [0,T], \label{eq:multil1} \\
    &\norma{\rho_i(t_1,\cdot)-\rho_i(t_2,\cdot)}_1
    \leq \mathcal{K}_i \modulo{t_1 -t_2} &\hbox{for } t_1,t_2\in\, [0,T], \label{eq:multiL1time} 
\end{align}
 for $\mathcal{K}_i$ defined in \eqref{k}, and such that for  $t\in [0,T]$
 \begin{equation}
\sum_{i=1}^M\tv(\rho_i(t,\cdot))\leq\mathcal{B}_t\prod_{j\in\tilde{\mathcal{J}}\atop t\geq\tau_j}\mathcal{B}_j\sum_{i=1}^M\tv(\rho_i^0)\label{eq:multiTV} 
 \end{equation}
where the positive constants $\mathcal{B}_j$ and $\mathcal{B}_t$ are given in \eqref{Bj} and \eqref{BT} and the set $\tilde{\mathcal{J}}$ is defined in \eqref{Jtilde}.
\end{theorem}

\begin{proof}
    Lemma \ref{multiboundteo} and Proposition \ref{multiBVteo} ensure from Helly's Theorem that the approximate solution $\boldsymbol\rho^{\Delta x}$ converges in the $\Lloc{1}$-norm  up to subsequences to some $\rho_i\in\BV([0,T]\times\R ;[0,R_i])$ as $\Delta x \searrow 0$.
By applying the classical procedure of Lax-Wendroff theorem and following closely \cite[Theorem 1.2]{ChiarelloGoatin2019}, one can prove that the limit function $\boldsymbol\rho=\left(\rho_1,\dots,\rho_M\right)$ is an entropy weak solution of~\eqref{multiclasse}-\eqref{eq:initial_datum} in the sense of Definition \ref{multientropy}. The complete proof is in Appendix \ref{sec:app3}.
\end{proof}

\begin{remark}
{\rm
For completeness, we analyze the specific case where $M$ classes share the same delay $\tau > 0$, velocity function $v$, maximum density $R$, and saturation function $f$. Additionally, we assume that $J_1=\dots=J_M=:J_0$. In this scenario, for $t \geq \tau$, the bound \eqref{eq:multiTV} simplifies to 
    $$
    \sum_{i=1}^M\tv(\rho_i(t,\cdot))\leq\left((M+1)e^{\mathcal{M}\left(t-\lfloor t/\tau\rfloor\tau\right)}-M\right)\left((M+1)e^{\mathcal{M}\tau}-M\right)^{\lfloor T/\tau\rfloor}\sum_{i=1}^M\tv(\rho_i^0),
    $$
    where
    $$
    \mathcal{M}=2R\max\left\{\left(3\norma{v''}J_0MR+\norma{v'}\right),M\norma{v'}\left(1+R\norma{f'}\right)\right\}\max_{1\leq i\leq M}\norma{\omega_i},
    $$
which is the same estimate than for the scalar case \cite{CiaramagliaGoatinPuppo2024}.    
We note, as discussed in Remark \ref{relazioneoss}, that $\mathcal{M} \sim 1/L_i$, which further demonstrates that the larger the distance drivers can perceive, the more stabilized the traffic becomes.
}\end{remark}

By properly adapting Kru{\v{z}}kov's doubling of variables technique \cite{Kruzkov}, we now prove a stability result for entropy weak solutions of~\eqref{multiclasse}-\eqref{eq:initial_datum}, from which
follows uniqueness. Below, we denote by $\norma{\boldsymbol\rho(t,\cdot)}_1:=\sum_{i=1}^M\norma{\rho_i(t,\cdot)}_1$ the $\L1$- norm 
    in $\L1(\R;\R^M)$
    and by $\norma{\boldsymbol\tau}_1:=\sum_{i=1}^M\modulo{\tau_i}$ the $\L1$-norm in $\R^M$.

\begin{theorem}[$\L1$ stability]\label{stabilityteo}
    Given Assumption \ref{multi hp}, let $\boldsymbol\rho$ and $\boldsymbol\sigma$ be two entropy weak solutions of \eqref{multiclasse}-\eqref{eq:initial_datum} as in Definition \ref{multientropy}, with initial data $\boldsymbol\rho^0,\boldsymbol\sigma^0\in\BV(\R;[0,R_1]\times\dots\times[0,R_M])$ and delays $\boldsymbol\tau$ and $\boldsymbol\nu$, respectively. Then, for any $T>0$ there holds
    \begin{equation}\label{stimaL1}
        \norma{\boldsymbol\rho(t,\cdot)-\boldsymbol\sigma(t,\cdot)}_{1}\leq e^{K_1T}
        \left(K_3\norma{\boldsymbol\rho^0-\boldsymbol\sigma^0}_{1}
        + K_2\norma{\boldsymbol\tau-\boldsymbol\nu}_1
        \right)
        \qquad\forall t\in[0,T],
    \end{equation}
    with $K_1,K_2,K_3>0$ given by \eqref{Ki}.
\end{theorem}

\begin{proof}
    We proceed as in~\cite[Lemma 4 and Proof of Theorem 1]{ChiarelloGoatinRossi2019}.
    The functions $\boldsymbol\rho=(\rho_1,\dots,\rho_M)$ and $\boldsymbol\sigma=(\sigma_1,\dots,\sigma_M)$ are respectively entropy weak solutions of the equations
    \begin{eqnarray*}
        \partial_t\rho_i(t,x)+\partial_x\big(\rho_i(t,x)f_i(\rho_i(t,x))\mathcal{V}_i(t-\tau_i,x)\big)=0,&&\mathcal{V}_i(t,x):=v_i\left(\left(r_\rho\ast\omega_i\right)(t,x)\right),\\
        \partial_t\sigma_i(t,x)+\partial_x\big(\sigma_i(t,x)f_i(\sigma_i(t,x))\mathcal{U}_i(t-\nu_i,x)\big)=0,&&\mathcal{U}_i(t,x):=v_i\left(\left(r_\sigma\ast\omega_i\right)(t,x)\right),
    \end{eqnarray*}
for $i=1,\dots,M$, where we denoted $r_\rho=\sum_{l=1}^M\rho_l$ and $r_\sigma=\sum_{l=1}^M\sigma_l$. Moreover, they fulfill the following initial conditions
    \begin{align*}
        \rho_i(t,x)&=\rho_i^0(x),\qquad\mbox{ for }(t,x)\in\left[-\norma{\boldsymbol\tau},0\right]\times\R ,\\
        \sigma_i(t,x)&=\sigma_i^0(x),\qquad\mbox{ for }(t,x)\in\left[-\norma{\boldsymbol\nu},0\right]\times\R .
    \end{align*}
Following \cite[Proof of Theorem 4.3]{CiaramagliaGoatinPuppo2024}, we get
\begin{align*}
    \norma{\rho_i(T,\cdot)-\sigma_i(T,\cdot)}_{1}\leq\ &\norma{\rho^0_i-\sigma^0_i}_1+Tk_i\modulo{\tau_i-\nu_i}\sum_{l=1}^M\mathcal{K}_l\\
    &+k_i\min\{\tau_i,\nu_i\}\norma{\boldsymbol{\rho}^0-\boldsymbol{\sigma}^0}_1+k_i\int_0^T\norma{\boldsymbol\rho(t,\cdot)-\boldsymbol\sigma(t,\cdot)}_1\d t,
\end{align*}
where
\begin{align*}
        k_i=&\ \norma{\omega_i}\norma{v_i'}\left(1+R_i\norma{f_i'}\right)\sup_{t\in[0,T]}\norma{\rho_i(t,\cdot)}_{\BV(\R )}\\
        &+\norma{v'_i}\left(\norma{\rho_i^0}_1\norma{\partial_x\omega_i}_{\L\infty(0,L_i)}+2R_i\norma{\omega_i}\right)+\norma{\rho_i^0}_1\norma{v''_i}\norma{\omega_i}^2\sum_{l=1}^M\sup_{t\in[0,T]}\norma{\rho_l(t,\cdot)}_{\BV(\R)}.
\end{align*}
Summing over $i=1,\dots,M$, we obtain
\begin{align*}
\norma{\boldsymbol\rho(T,\cdot)-\boldsymbol\sigma(T,\cdot)}_1\leq K_3\norma{\boldsymbol\rho^0-\boldsymbol\sigma^0}_1+K_2\norma{\boldsymbol\tau-\boldsymbol\nu}_1+K_1\int_0^T\norma{\boldsymbol\rho(t,\cdot)-\boldsymbol\sigma(t,\cdot)}_1\d t,
\end{align*}
with 
\begin{equation}\label{Ki}
K_1=\sum_{i=1}^Mk_i,\qquad K_2=T\left(\max_{1\leq i\leq M}k_i\right)\sum_{l=1}^M\mathcal{K}_l,\qquad K_3=1+\sum_{i=1}^Mk_i\min\{\tau_i,\nu_i\}.
\end{equation}
The statement follows from the Gronwall's lemma.
\end{proof}

 Theorem~\ref{stabilityteo} also states stability w.r.t. the delay $\boldsymbol\tau=\left(\tau_1,\dots,\tau_M\right)$, which as a byproduct gives the following convergence result
 (see also~\cite[Corollary 4.4]{CiaramagliaGoatinPuppo2024}). 

\begin{cor}[Convergence for delay tending to zero]\label{stabilitycor}
    Let Assumption \ref{multi hp} hold. Given $\boldsymbol\rho^0\in\BV(\R;[0,R_1]\times\dots\times[0,R_M])$, let $\boldsymbol\rho_{\boldsymbol\tau}\in\L1([0,T]\times\R;\R^M)$ denote the solution of the Cauchy problem \eqref{multiclasse}-\eqref{eq:initial_datum} for $\norma{\boldsymbol\tau}>0$. Let also $\boldsymbol\rho\in\L1([0,T]\times\R;\R^M)$ be the entropy solution of
    \begin{equation}\label{nodelay_bis}
    \partial_t \rho_i(t,x)+\partial_x\big(\rho_i(t,x)f_i(\rho_i(t,x))v_i((r\ast \omega_i)(t,x))\big)=0,\qquad i=1,\dots,M,
    \end{equation}
    with the same initial condition.
    Then, we have the convergence
    $$
    \norma{\boldsymbol\rho_{\boldsymbol\tau}(t,\cdot)-\boldsymbol\rho(t,\cdot)}_{\L1}\rightarrow 0,\qquad\forall t\in[0,T],
    $$
    as $\norma{\boldsymbol\tau}_1\rightarrow 0$.
\end{cor}

\section{Numerical tests}\label{sec:numerics}
This section is devoted to present some numerical simulations illustrating the features of the multi-class non-local model with time delay~\eqref{multiclasse}. For simplicity and aiming at modeling HV-AV interactions, we limit the study to  $M=2$ classes of vehicles, thus considering
\begin{equation}\label{biclasse}
\begin{cases}
\partial_t \rho_1(t,x)+\partial_x\left(\rho_1(t,x)f_1(\rho_1(t,x))v_1((r\ast \omega_1)(t-\tau_1,x))\right)=0,\\[10pt]
\partial_t \rho_2(t,x)+\partial_x\left(\rho_2(t,x)f_2(\rho_2(t,x))v_2((r\ast \omega_2)(t-\tau_2,x))\right)=0,
\end{cases}
\end{equation}
where $r=\rho_1+\rho_2$.

For the tests, we consider a ring road of length $l=2$ populated by two classes of vehicles, describing a scenario of heterogeneous traffic flow.
The space domain is given by the interval $[0,2]$ equipped with periodic boundary conditions and the space discretization mesh is $\Delta x=5\cdot 10^{-3}$. Moreover, we consider relative densities setting $R_1=R_2=1$ for simplicity.

\subsection{The effect of saturation}
As remarked previously,  the presence of the saturation functions $f_i$ guarantees that, if $\brho^0(x)\in[0,R_1]\times[0,R_2]$, then each class density $\rho_i$ never exceeds the corresponding maximal value $R_i$ and satisfies 
\begin{equation}\label{limitecomponente}
0\leq\rho_i(t,x)\leq R_i,\qquad\forall x\in \R,~ t>0.
\end{equation}
 This property does not hold in general without saturation, see~\cite[Lemma 2.2]{ChiarelloGoatin2019} and~\cite[Corollary 4.4]{KeimerPflug2019}. 
 \\To illustrate these characteristics, we consider a scenario with a slow and a fast class of vehicles on the road. To model this situation, we choose the widely used Greenshields' velocity function \cite{Greenshields}
\begin{equation}\label{Greenshields}
v_i(r)=V_i\left(1-\frac{r}{R_i}\right),\qquad\mbox{ for }i=1,2,
\end{equation}
with $V_1=0.04$ and $V_2=0.015$. We assume that the faster class $\rho_1$ is initially positioned behind the slower class on the road taking
\begin{equation}\label{data_trasl}
\begin{cases}
\rho_1^0(x)=\frac{8}{9}\operatorname{exp}\left\{-100\left(x-\frac{1}{4}\right)^2\right\},\\[10pt]
\rho_2^0(x)=\frac{8}{9}\operatorname{exp}\left\{-100\left(x-\frac{9}{10}\right)^2\right\},
\end{cases}
\end{equation}
so that an overtaking is expected to occur for sufficiently large times. This initial setup is designed to simulate a critical scenario where vehicles are closely spaced, leading to reduced flow and potential traffic jams. Moreover, it guarantees that, initially, both populations contribute equally to the overall traffic density. The initial data \eqref{data_trasl} is represented component-by-component in the graphs at the top of Figure~\ref{fig:sat}.
For simplicity, we consider equal constant kernels $\omega_i(x)=1/L_i$ with $L_1=L_2=0.1$, and we assume that both classes have the same delay $\tau_1=\tau_2=2.5$.
As saturation function, we consider the exponential decreasing function
\begin{equation}\label{expf}
f_i(\rho)=1-e^{50(\rho-R_i)},\qquad i=1,2,
\end{equation}
which was taken in~\cite{CiaramagliaGoatinPuppo2024} as an approximation of the characteristic function $\chi_{[0,R_i)}$, thus acting only on density values close to the maximal density.
We remark that the choice of the saturation functions has an impact on all the estimates provided in the previous sections through $\norma{f_i'}$.
\begin{figure}
\centering
{\includegraphics[width=17cm]{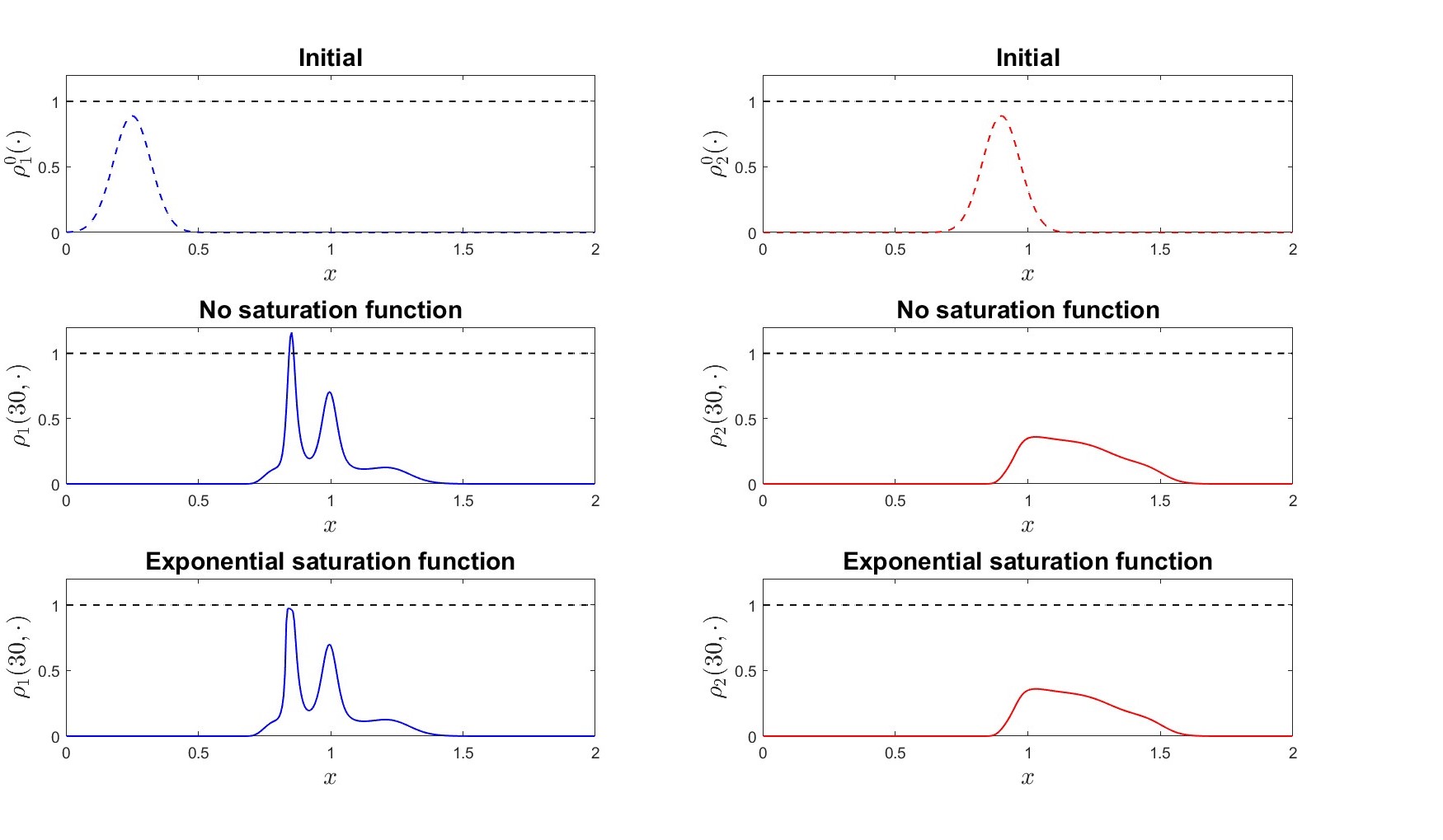}}
\caption{Comparison between the solution of the model \eqref{biclasse} with no saturation ($f_i\equiv 1$ for $i=1,2$) and the solution corresponding to the saturation functions \eqref{expf}. The initial data is given by \eqref{data_trasl} and the final time is $T=30$. \textbf{ Left column}: density $\rho_1(T,\cdot)$ of fast cars ($V_1=0.04$); \textbf{Right column}: density $\rho_2(T,\cdot)$ of slow cars ($V_2=0.015$).
}\label{fig:sat}
\end{figure}

In Figure \ref{fig:sat}, we compare each component of the solutions of the model \eqref{biclasse} with no saturation (i.e. setting $f_i\equiv 1$ for $i=1,2$)  and  with saturation given by \eqref{expf}, at the final time $T=30$, when the overtaking is occurring. In the left column, we display the first component of the solution $\rho_1$, while the second component $\rho_2$ is presented in the right column of the figure.
As expected, the saturation term constrains the corresponding component $\rho_i(t,x)$ in the interval $[0,R_i]$ (this holds in particular for $i=1$, which violates the maximal density bound in the absence of saturation).
%

\subsection{Study of the invariant domain}\label{saturationtotal}
In this section we aim to numerically compare the original model \eqref{biclasse} with the modified model 
\begin{equation}\label{mod_biclasse}
\begin{cases}
\del_t\rho_1(t,x)+\del_x\left(\rho_1(t,x)f_1(r(t,x))v_1((r\ast\omega_1)(t-\tau_1,x))\right)=0,\\[5pt]
\del_t\rho_2(t,x)+\del_x\left(\rho_2(t,x)f_2(r(t,x))v_2((r\ast\omega_2)(t-\tau_2,x))\right)=0,
\end{cases}
\end{equation}
discussed in Remark~\ref{rem:invariant_domain} for a general number of classes. In Lemma~\ref{multiboundteo_mod}, we proved that, for~\eqref{mod_biclasse}, the simplex
$$
\mathcal{S}:=\left\{\boldsymbol\rho\in\R^2\left|\right.\rho_1+\rho_2\leq 1,~\rho_i\geq 0~\mbox{ for }i=1,2\right\}
$$
is invariant, unlike the model without saturation studied in~\cite{ChiarelloGoatin2019}. 
\begin{figure}
\centering
{\includegraphics[width=16cm]{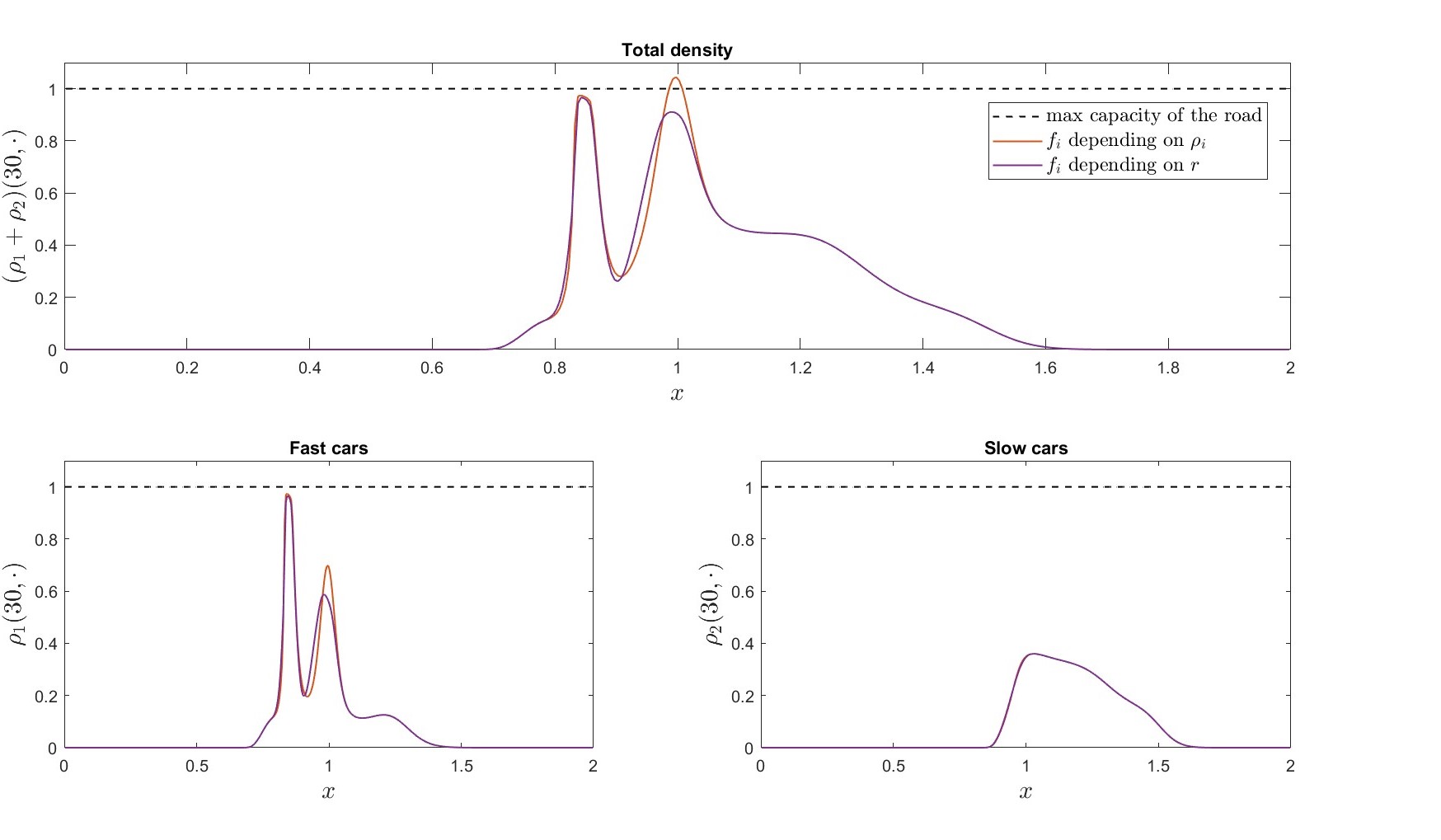}}
\caption{Comparison between the model~\eqref{biclasse} considered in this work, and the modified model~\eqref{mod_biclasse}, where the saturation functions depend on the total density $r=\rho_1+\rho_2$. The initial data is \eqref{data_trasl}, the kernels are constant and the saturation functions are both exponential functions. \textbf{Top row:} Total density. \textbf{Bottom row:} Singular densities taken individually. } 
\label{fig:mixsatfigure_total_real}
\end{figure}
Even if well-posendess results for~\eqref{mod_biclasse} are currently missing, 
numerical simulations allow to compare it with~\eqref{biclasse}. In Figure~\ref{fig:mixsatfigure_total_real}, we consider the same scenario as in the previous example and we present the solution at the final time $T=30$ for both the original and the modified models. The numerical results show that, although the individual densities are always constrained between $0$ and $1$, for model~\eqref{biclasse} the total density $r$ exceeds $1$, thus violating  the road’s maximal capacity, while the solution of the modified model~\eqref{mod_biclasse} satisfies $0\leq r(t,x)\leq 1$, consistently with Lemma~\ref{multiboundteo_mod}.
However, in the following tests, we will consider class-specific exponential saturation functions as in~\eqref{expf}, for which we proved well-posedness.

\begin{figure}
\centering
{\includegraphics[width=16cm]{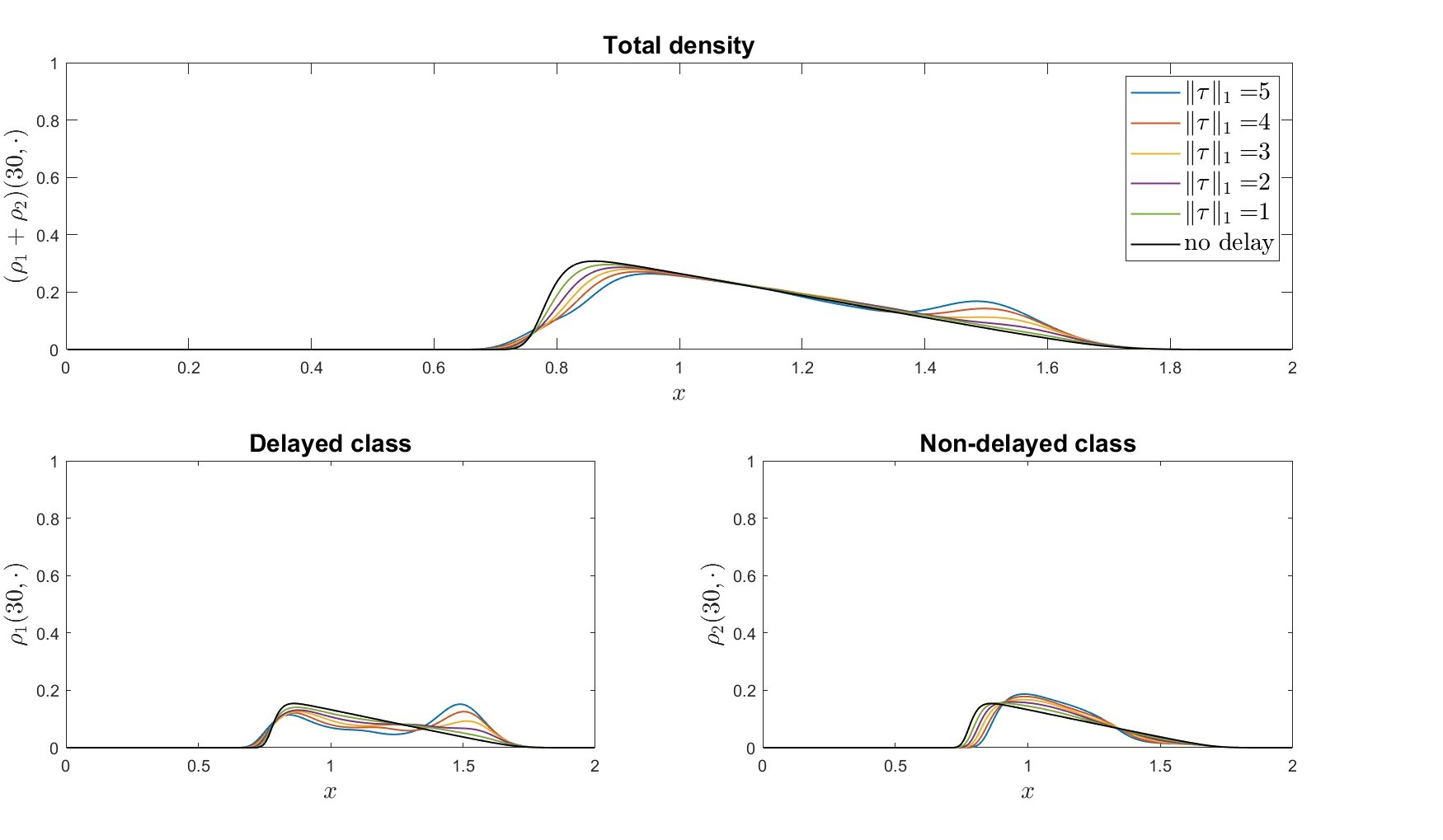}}
\caption{Convergence of the delayed model \eqref{biclasse} with initial data $\rho_1^0(x)=\rho_2^0(x)=\frac{1}{2}r^0(x)$ and initial total density given by \eqref{eq:tests} to the model with no delay \eqref{nodelay_bis} with the same initial data, as $\norma{\boldsymbol\tau}_1\rightarrow 0$. \textbf{Top row}: total density $\rho_1+\rho_2$; \textbf{Bottom row}: densities $\rho_1$ and $\rho_2$ plotted individually.}\label{convergencedelay}
\end{figure}

\subsection{Convergence to the non-delayed model}\label{convergencesec}

The aim of this test is to illustrate the convergence of the solution when the delay tends to zero, which was stated in Corollary~\ref{stabilitycor}. 
We set $\tau_2=0$, thus quantifying the impact of delay only through the parameter $\norma{\boldsymbol\tau}_1=\tau_1$.
In order to ensure that both vehicle classes exert a balanced influence on the system, we assume that the initial density for each class is given by $\rho_1^0(x)=\rho_2^0(x)=\frac{1}{2}r^0(x)$, with
\begin{equation}
r^0(x)=\frac{8}{9}\operatorname{exp}\left\{-100\left(x-\frac{1}{4}\right)^2\right\}.\label{eq:tests}
\end{equation}
For the same reason, we choose for the two populations the linear speeds \eqref{Greenshields} with same maximum speed $V_1=V_2=0.04$ and we take constant weight kernels $\omega_i=1/L_i$, with $L_1=L_2=0.1$. \\
In the top line of Figure~\ref{convergencedelay}, we plot the total densities  at the final time $T=30$ corresponding to decreasing values of $\tau_1\in\{5,4,3,2,1,0\}$. We can observe that, as $\tau_1$ decreases, the solution approaches the solution of the non-delayed model~\eqref{nodelay_bis}. 
Additionally, as further evidence that the delay contributes to the solution's instability, we note that higher values of $\tau_1$ correspond to increasingly oscillatory density profiles. The bottom row displays the two density components taken separately, confirming the result.

\subsection{Introducing AVs to enhance traffic flow dynamics}\label{AVHV}
As anticipated in the Introduction, our main motivation for introducing the multi-class model~\eqref{biclasse} was to study the mutual interactions between HVs and AVs. Indeed, some studies \cite{AvedisovBansalOrosz2022,GBEMAH2016,LEVIN2016103,TALEBPOUR2016143,YE2018269} have shown that even a small number of controlled vehicles are able to regulate traffic flow, dampening stop-and-go waves and reducing fuel consumption and pollutant emissions. 
To verify that our model is consistent with these results, we consider a circular road populated by HVs and AVs, 
whose densities are denoted by $\rho_H$, $\rho_A$, respectively.
Correspondingly, we set the time delay parameters as
\begin{equation}\label{hp tau}
\tau_H=2.5\quad\quad\mbox{ and }\quad\quad\tau_A=0.
\end{equation}
Moreover, we assume the AVs can count on a significant look-ahead distance while humans are subject to their biological limits, which implies $L_H<L_A$. Thus we fix
\begin{equation}\label{hp L}
L_H=0.1 \quad\quad\mbox{ and }\quad\quad L_A=0.2.
\end{equation}
\begin{figure}
\centering
{\includegraphics[width=16cm]{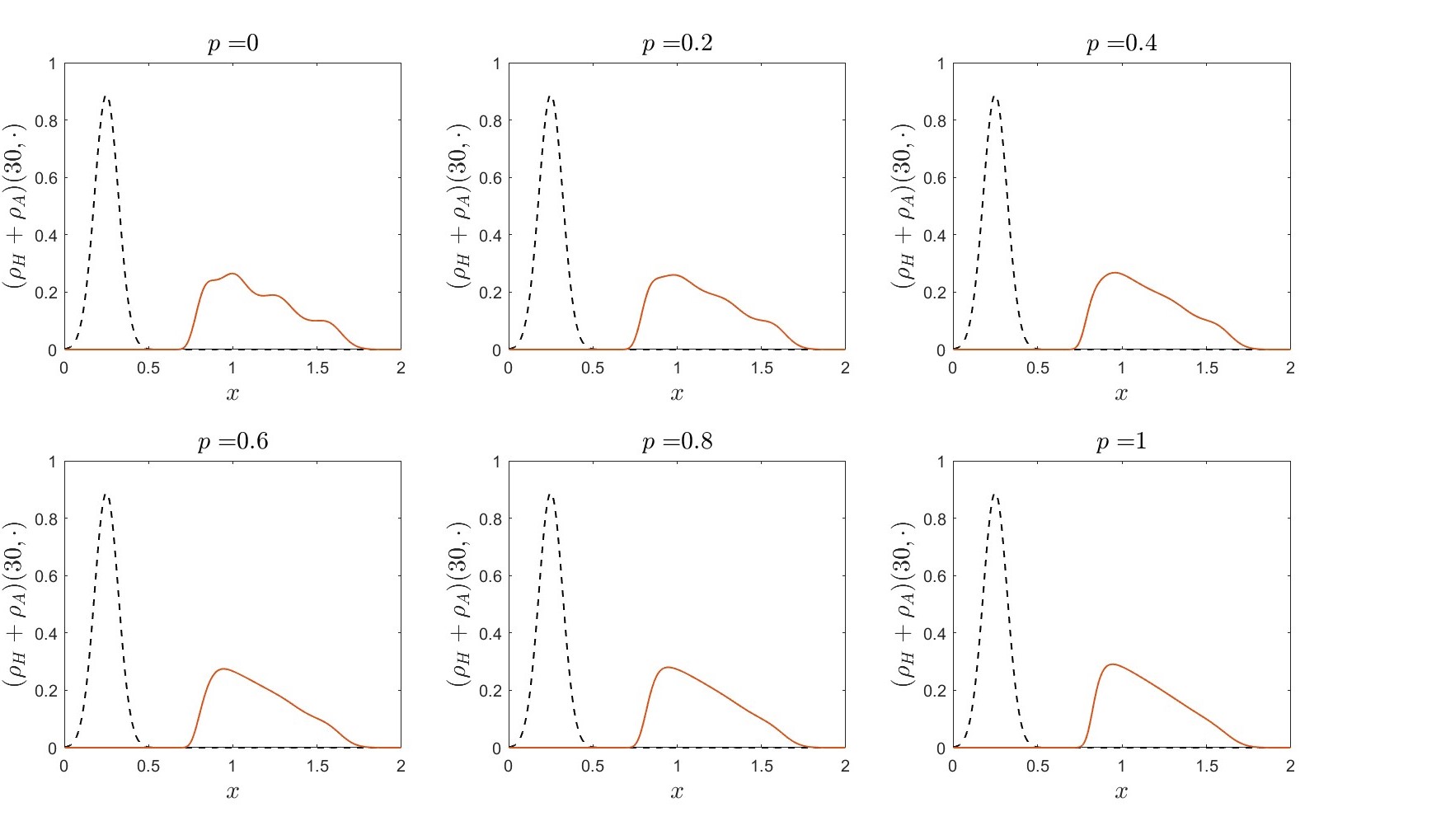}}
\caption{Comparison between the total density of the solution of~\eqref{biclasse} with initial data \eqref{eq:tests}-\eqref{initial1}-\eqref{initial2} and respectively constant (AVs) and linear decreasing (HVs) kernels, corresponding to different values of the penetration rate $p\in[0,1]$.}\label{confrontop}
\end{figure} 
\begin{figure}
\centering
{\includegraphics[width=16cm]{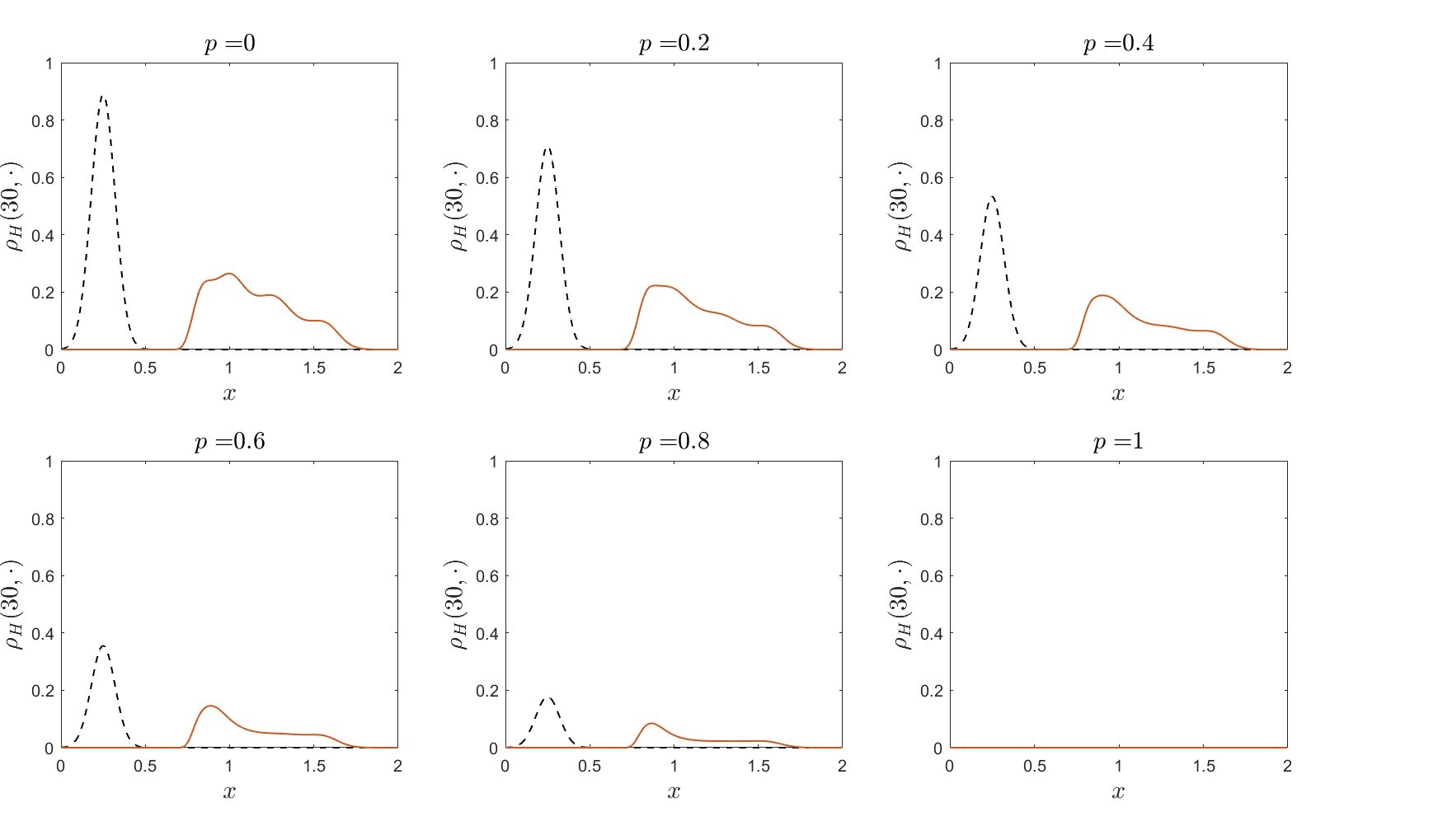}}
{\includegraphics[width=16cm]{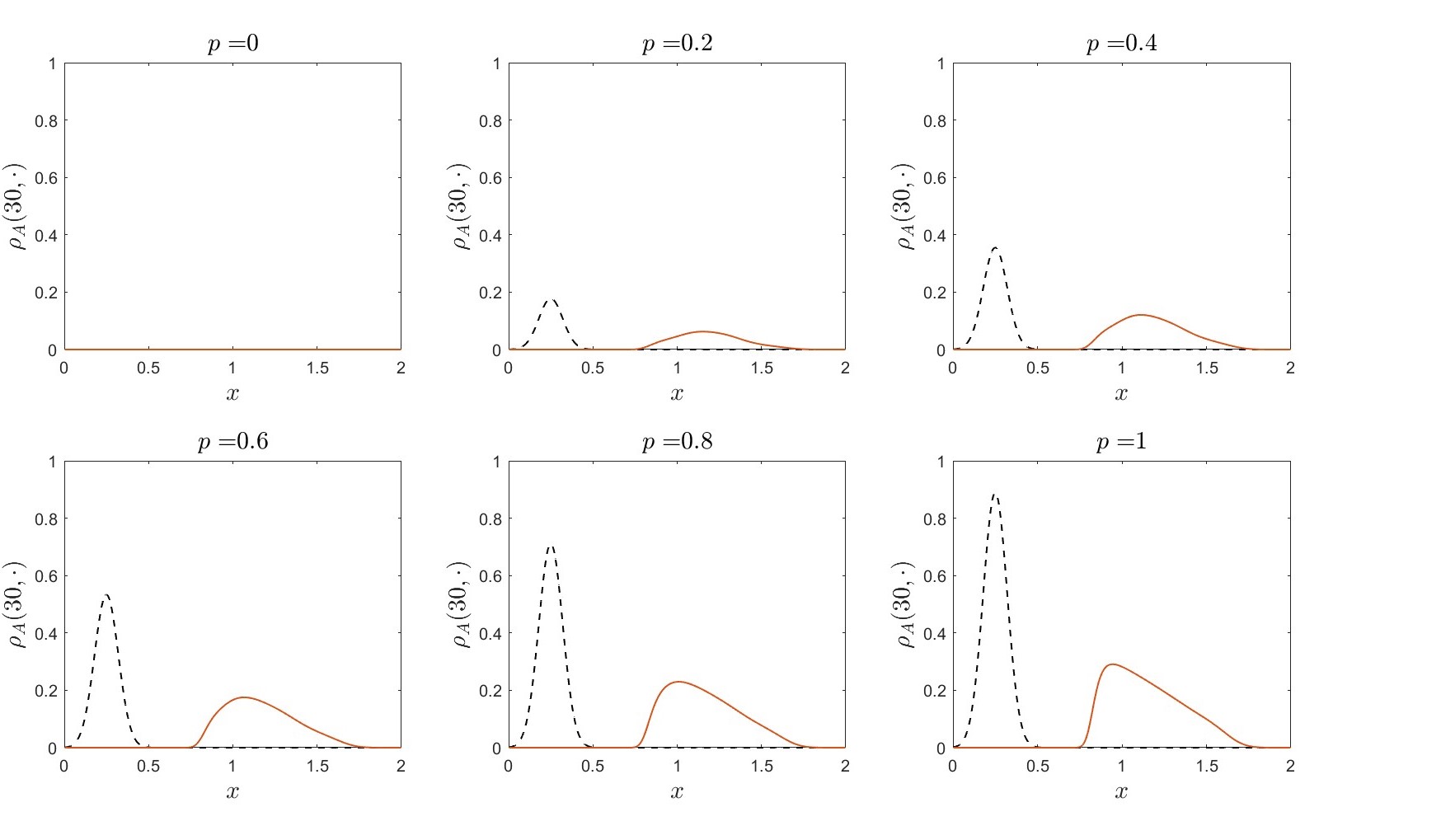}}
\caption{Density profiles of each class taken individually, corresponding to the total densities shown in Figure~\ref{confrontop}. \textbf{Two top rows:} HVs. \textbf{Two bottom rows:} AVs. }\label{confrontopsep}
\end{figure}
As in the previous section, we start choosing  Greenshields' velocity function \eqref{Greenshields} with $V_H=V_A=0.04$. Regarding the weight kernels, we remark that the choice of a constant kernel for the AVs models the fact that they may be able to have the same degree of accuracy on information about surrounding traffic, independently from the distance. On the other hand, it is reasonable to assume that the weight kernel associated to human drivers decreases with the distance. Thus, as in~\cite{ChiarelloGoatin2019}, we choose
\begin{equation}\label{omega}
\omega_H(x)=\dfrac{2}{L_H}\left(1-\dfrac{x}{L_H}\right)\qquad\mbox{and}\qquad\omega_A(x)=\dfrac{1}{L_A}.
\end{equation}
\begin{figure}
\centering
{\includegraphics[width=16cm]{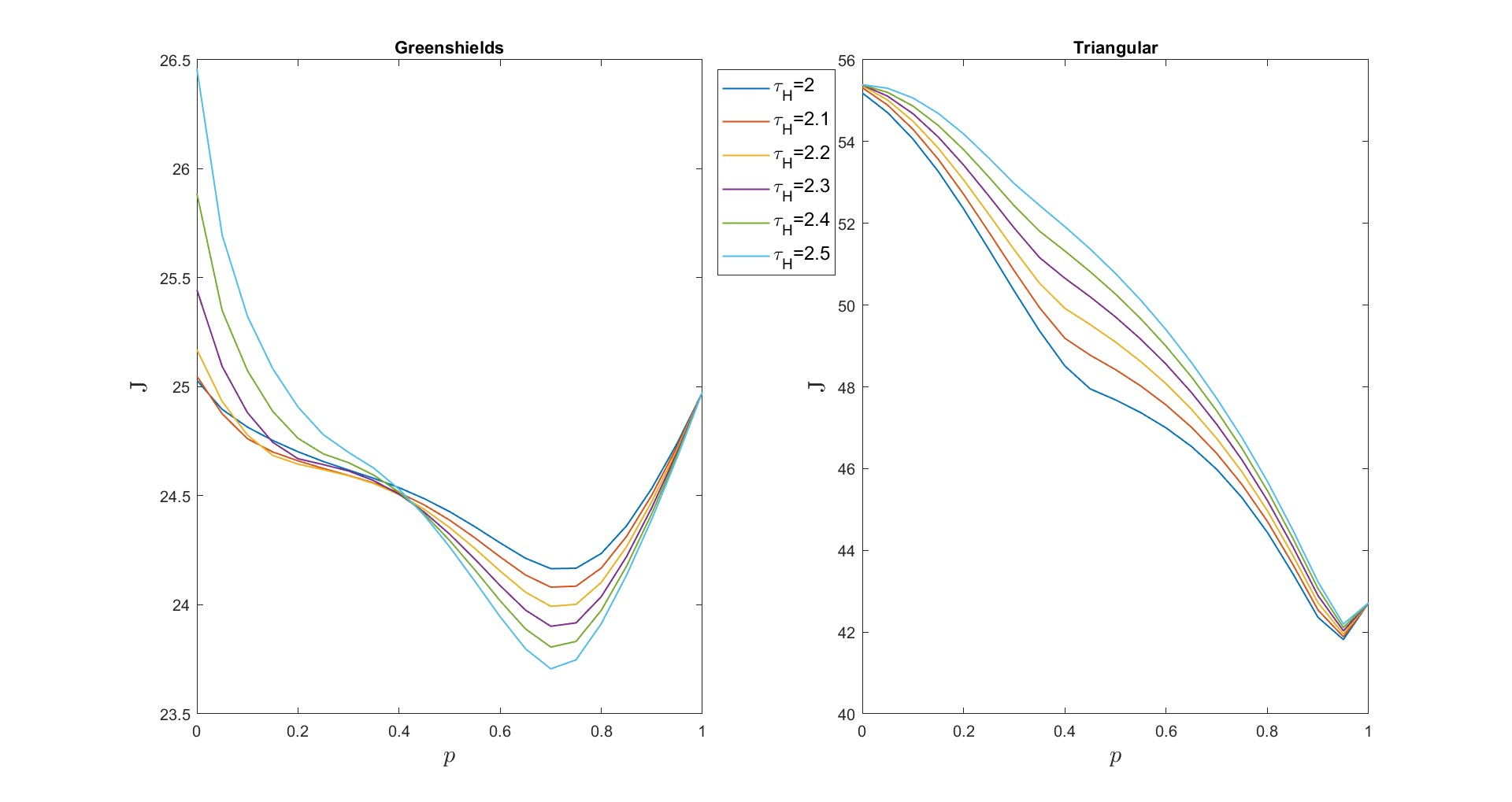}}
\caption{Functional $J$ defined in \eqref{Jfun} for $T=30$ associated to the initial datum \eqref{eq:tests}-\eqref{initial1}-\eqref{initial2} and to the delay $\tau_H\in\{2,2.1,2.2,2.3,2.4,2.5\}$. \textbf{Left:} Greenshields' speed law \eqref{Greenshields}. \textbf{Right:} Triangular speed law \eqref{trispeed}.}\label{functionals}
\end{figure}
Aiming to evaluate the impact of AVs on the overall traffic flow, we consider different \textit{penetration rates} $p\in[0,1]$ indicating the percentage of AVs in the total traffic. More precisely, given \eqref{eq:tests} as an initial condition for the total density $r^0(x):=\rho_H(0,x)+\rho_A(0,x)$, we set classes' initial data as
\begin{subequations} \label{eq:initial}
    \begin{align}
        \label{initial1}
\rho_H(0,x)&=(1-p)r^0(x), \\
\label{initial2}
\rho_A(0,x)&=p~r^0(x).
    \end{align}
\end{subequations}
 In particular, $p=1$ corresponds to a completely autonomous fleet, while $p=0$ represents the case where no AVs are present.
In Figures~\ref{confrontop} and~\ref{confrontopsep}, we compare the numerical solutions (respectively, the total density and the individual density components) of~\eqref{biclasse} with initial data~\eqref{eq:tests}-\eqref{initial1}-\eqref{initial2}, corresponding to different penetration rates $p=0, 0.2, 0.4, 0.6, 0.8, 1$.
We can see that in the fully human-driven situation, shown in the top left plot, the solution develops some oscillations, which reduce as $p$ increases. In particular, Figure~\ref{confrontopsep} shows that these instabilities are due only to the HV component.\\ 
\begin{figure}
\centering
{\includegraphics[width=8cm]{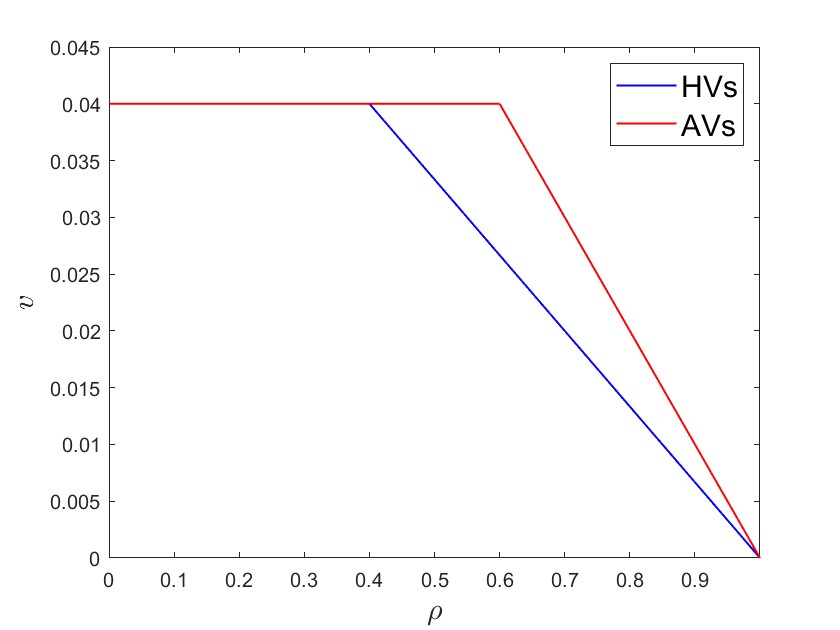}}
\caption{Speed-density relation described in \eqref{trispeed} with $V_H=V_A=0.04$, and $\rho_{c,H}=0.4$ and $\rho_{c,A}=0.6$.}\label{vel fig}
\end{figure}
The graphs in Figures~\ref{confrontop},~\ref{confrontopsep} show the density profiles at time $T=30$. To better investigate the stabilizing effect of the presence of AVs during the whole time interval $[0,T]$, we consider the functional
\begin{equation}\label{Jfun}
    J(p)=\int_0^T \d{} \modulo{\partial_x r} \d t,
\end{equation}
which measures the integral with respect to time of the spatial total variation of the total traffic density in the time interval $[0,T]$, see also~\cite{ChiarelloGoatin2019}. 
In Figure~\ref{functionals} (left), we plot the functional $J$ associated to the solutions of the test corresponding to  Figures~\ref{confrontop} and~\ref{confrontopsep}, for various delay values.
 As expected, in the limit case $p=1$ the profile is independent of $\tau_H$, while when $p=0$ the functional $J$ assumes values increasing with the delay. In addition, we observe that $J$ is in general not monotone with respect to the penetration rate $p$ and takes optimal (minimum) values close to $p=0.7$, consistently with the non-delayed model, see~\cite[Figure 4]{ChiarelloGoatin2019}. In particular, the functional is monotonically decreasing for $p$ not too close to $1$, showing the stabilizing effect of the AVs. Concerning the delay, we see that the bigger is $\tau_H$, the steeper is the profile, thus the stronger the stabilization capacity of AVs.



In Figure~\ref{functionals} (right), we repeated the analysis choosing a different speed law than the Greenshields' velocity function \eqref{Greenshields}. 
Various speed functions are proposed in the literature. In particular,  two different regimes are usually identified in the fundamental diagram: the free flow  and congested regimes, see e.g.~\cite{LEVIN2016103,YE2018269,ZHOU2020102614}. When in a free-flow state, vehicles are supposed to travel at the maximal speed limit. On the other hand, once in congestion, the velocity of traffic decreases with the density until it reaches zero when the road is fully congested. 
Assuming a linear decrease in the congested branch of the fundamental diagram, in agreement with Greenshields' velocity function, we consider the  following speed law
\begin{equation}\label{trispeed}
v_i(r)=
\left\{
\begin{array}{cr}
V_i,&\mbox{ if }r\leq\rho_{c,i},\\
\frac{V_i}{\rho_{c,i}-R_i}(r-R_i),&\mbox{ otherwise,}\\
\end{array}
\right.
\qquad i=H,A.
\end{equation}
The critical densities $\rho_{c,i}\in[0,R_i)$, $i=H,A$, represent the transition points between free-flow and congestion. 
Since AVs have shorter reaction times compared to human drivers, they can safely maintain closer spacing between vehicles and thus they are able to keep free-flow speed for higher densities \cite{YE2018269,ZHOU2020102614}. 
For these reasons, \cite{LEVIN2016103} proposes a flow-density relationship as a function of reaction time, in which the capacity for free flow speed increases as reaction time decreases.
Consistently, we require
$$
\rho_{c,H}<\rho_{c,A},
$$
indicating that the presence of AVs can increase the road capacity.
Figure~\ref{vel fig} shows the speed functions~\eqref{trispeed} corresponding to $\rho_{c,H}=0.4$, $\rho_{c,A}=0.6$ and the same maximum speeds as before.
We recognize that these functions do not fulfill the smoothness hypothesis required by Assumption~\ref{multi hp}, which could be easily recovered by a smoothing process. However, this does not affect the numerical experiments.

Proceeding with the comparison between the Greenshields' and the triangular speed laws, in Figure~\ref{functionals} (right) we can see that the functional $J$ assumes much larger values when the triangular velocity function is used. 
Furthermore, the triangular speed law provides a more accurate representation of certain distinctive features of AVs behavior on the road, including their capability of stabilizing the overall traffic. Specifically, since human drivers enter a congested regime at lower densities compared to AVs, the stabilizing effect of AVs on traffic flow is stronger. This is clearly visible in Figure~\ref{functionals}, where the profile of the functional $J$ is steeper when the triangular velocity law~\eqref{trispeed} is used, indicating a sharper transition between the fully-human and the fully-autonomous regime, and thus highlighting the stronger capability of AVs to stabilize traffic. Since the right graph shows better results, we will adopt the triangular velocity law from now on.

\subsection{Oscillation dampening}

In the previous Section~\ref{AVHV}, we analyzed an AV-HV scenario assuming that the initial penetration rate of AVs is constant in space. In the following, we aim to extend this analysis by considering also a non-uniform distribution. Specifically, we investigate how a small perturbation in the initial conditions influences the system's temporal evolution. 
We consider model~\eqref{biclasse}, \eqref{trispeed} with critical densities $\rho_{c,H}=0.4$ and $\rho_{c,A}=0.6$,  and delay $\tau_H=2.0$ on a ring road (i.e. with periodic boundary conditions). \\
The baseline initial condition is given by the constant total density
$$
r^0(x)\equiv\rho^0=0.85,\qquad\mbox {for all }x\in [0,2],
$$
of which a fraction $p\in [0,1]$ is given by AVs, as in \eqref{eq:initial}. We apply to the initial data a perturbation $\theta(x)$ setting
\begin{equation}\label{perturbazione initial}
\rho_H(0,x)=\left[1-\left(p+\theta(x)\right)\right]\rho^0\qquad ,\qquad\rho_A(0,x)=\left(p+\theta(x)\right)\rho^0,
\end{equation}
as the class-specific initial conditions, 
where
\begin{equation}\label{theta}
\theta(x)=\frac{1}{30}\left[\cos\left(20\left(\frac{4}{3}x-\frac{1}{2}\right)\right)-\cos\left(10\left(\frac{4}{3}x-\frac{1}{2}\right)\right)\right]\chi_{\left[\frac{3}{20},\frac{(3\pi+1)}{20}\right]}(x)\,.
\end{equation}
Not that the initial total density is constant and, of course, if this perturbation wouldn't be present, the solution would remain constant. \\
The solution at the final time $T=30$ is represented in Figure~\ref{perturbazioni}, where we compare the solutions corresponding to different  penetration rates. In addition, in Figure~\ref{TV_perturbazioni}, we plot the total variation of the total density $r=\rho_H+\rho_A$ with respect to time for each simulation. 
We recall that the initial data was chosen according to~\eqref{eq:initial_datum}. This choice implies that, to observe the effect of delay on the stability of the solution, it is necessary to wait long enough. 
As a further proof of the smoothing impact of AVs, focusing on the right part of the graph, that is for large times, we can see that when the number of AVs increases, the perturbation is absorbed faster and the total variation is smaller.
\begin{figure}[h!]
\centering
{\includegraphics[width=17cm]{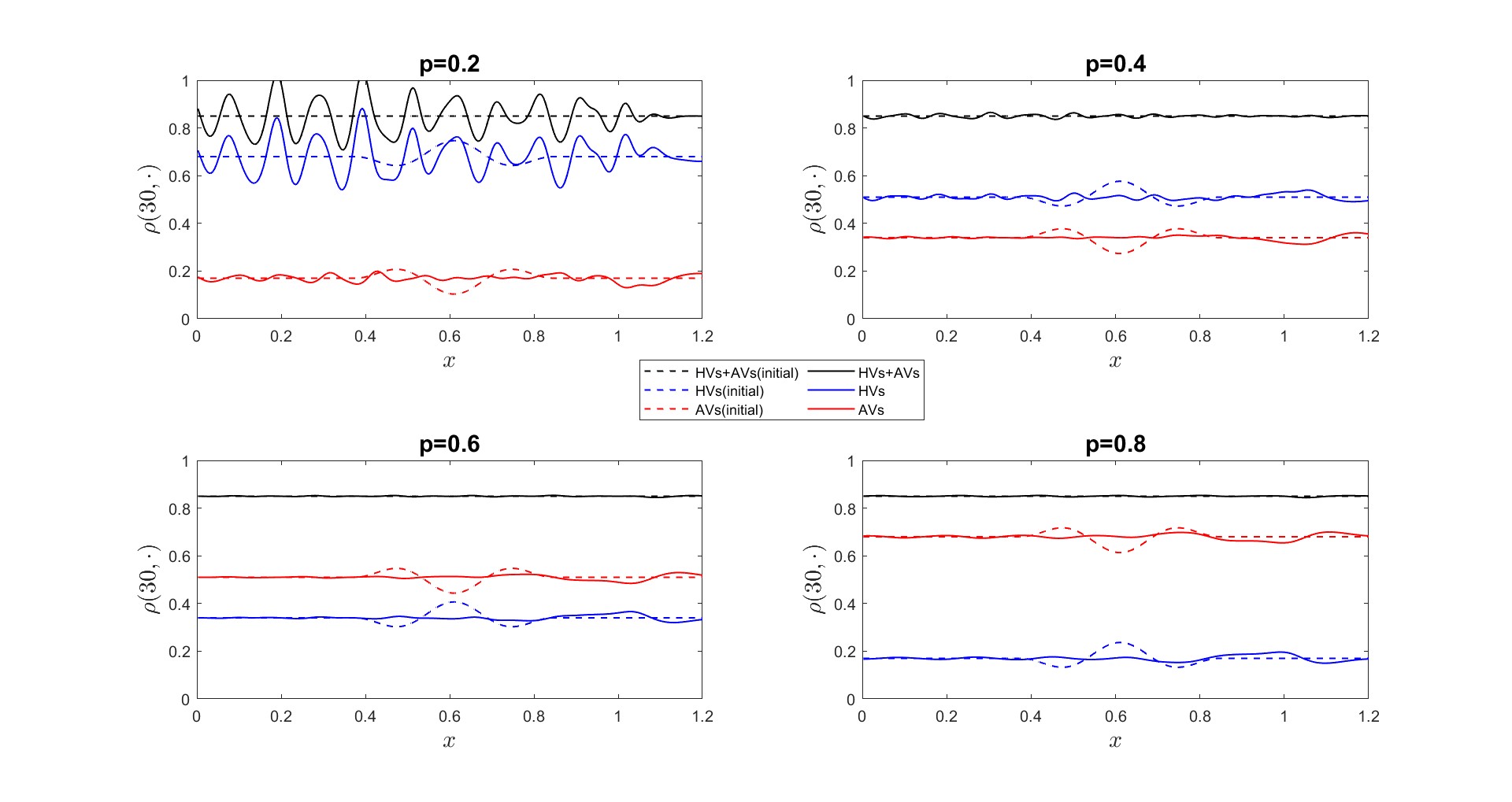}}
\caption{Solution at the final time $T=30$ of the model \eqref{biclasse} with initial condition given by \eqref{perturbazione initial} and penetration rate respectively equal to $p=0.2,0.4,0.6,0.8$.}\label{perturbazioni}
\end{figure}

\begin{figure}[h!]
\centering
{\includegraphics[width=9cm]{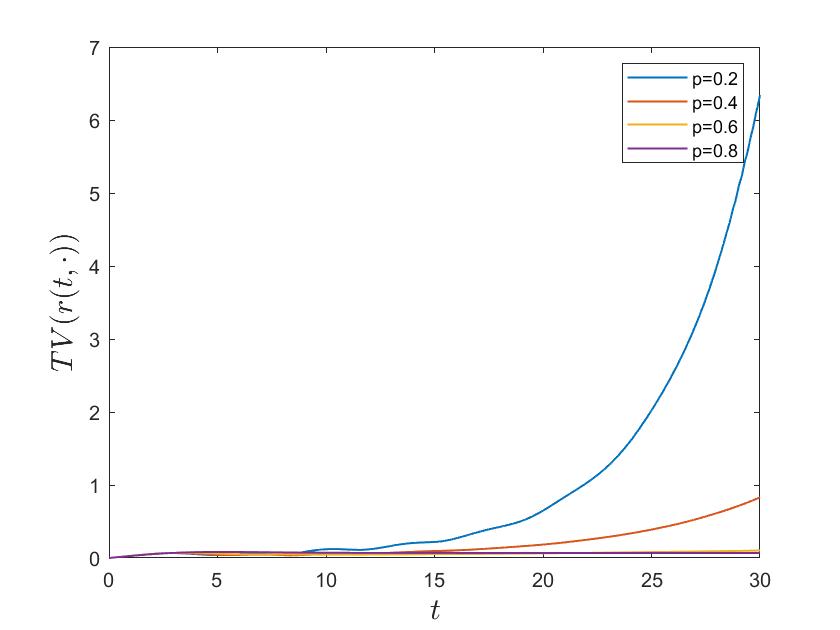}}
\caption{Total variation of the total density $r=\rho_H+\rho_A$ with respect to time associated to the tests in Figure \ref{perturbazioni}.}\label{TV_perturbazioni}
\end{figure}

\section{Conclusion}

In this paper, we introduced a non-local delayed  macroscopic model for multi-class traffic flow. We proved existence of solutions for $\BV$ initial data by showing that the Hilliges-Weidlich approximate solutions converge to an entropy weak solution of the model as the space discretization step tends to zero.  Additionally, we proved the $\L 1$ stability of solutions with respect to the initial data and the delay parameters by adapting Kru{\v{z}}kov's doubling of variables technique. This result implies the uniqueness of entropy weak solutions and their convergence to the associated non-delayed system.\\
This work extends the results obtained in the scalar case~\cite{CiaramagliaGoatinPuppo2024}, allowing for an investigation of heterogeneous traffic scenarios. It also improves the results obtained in~\cite{ChiarelloGoatin2019}, as the limit of our model, for all delays tending to zero, provides a non-local system of conservation laws for which all the properties, including global existence, are preserved. In contrast, for the model without saturation function considered in~\cite{ChiarelloGoatin2019}, existence of solutions is only guaranteed for small times, due to the blow-up of the $\L\infty$ bounds.

We then conducted a numerical analysis of the model, further investigating the effect of the saturation function in bounding each component of the solution under the relative maximum density. In addition, we investigated the impact of introducing AVs in a human-driven environment, showing that their presence can be beneficial in reducing traffic instabilities.


\appendix

\section{$\BV$ estimates: detailed computations} \label{sec: app2}
We detail here the computations that lead to the estimate \eqref{eq:TVestimate} for the general case $M>2$ (the case $M=2$ was developed in the proof of Proposition \ref{multispaceBVteo}). 
We consider the sequence described in \eqref{sequencetemp} and we assume that at least one of the time delay parameters is strictly positive, which means that the set $\mathcal{J}$ defined in \eqref{J} is non empty.
Thus, the sequence translates into
$$
\tv^{n+1}\leq(1+\Delta t\tilde{\mathcal{G}})\tv^n+\Delta t\mathcal{H}\sum_{i\in\mathcal{J}}\tv^{n-h_i},
$$
where $\tilde{\mathcal{G}}=\mathcal{G}+(M-\modulo{\mathcal{J}})\mathcal{H}$.
If we set $h_M:=\min_{j\in\mathcal{J}}h_j$, then as in \eqref{boundtemp} we get the bound
$$
\tv^k\leq\left((\modulo{\mathcal{J}}+1)(1+\Delta t\mathcal{M})^{h_M}-\modulo{\mathcal{J}}\right)\tv^0=:\mathcal{B}^{\Delta t}_M\tv^0,\qquad k=0,\dots,h_M,
$$
with $\mathcal{M}=\max\{\mathcal{H},\tilde{\mathcal{G}}\}$ defined in \eqref{eq:M}.
Regarding the following terms of the sequence, let us assume without loss of generality that the second smallest delay is given by $h_{M-1}:=\min_{j\in\mathcal{J}\atop h_j>h_M}h_j$. Then, for each $k=1,\dots,\big\lfloor\frac{h_{M-1}}{h_M}\big\rfloor$ we can write
\begin{align*}
    \tv^{(k-1)h_M+1}\leq&(1+\Delta t\tilde{\mathcal{G}})\tv^{(k-1)h_M}+\Delta t\mathcal{H}\left(\sum_{i\in\mathcal{J}\atop i\ne M}\tv^0+\tv^{(k-1)h_M}\right)\\
    \leq&(1+\Delta t\tilde{\mathcal{G}})(\mathcal{B}^{\Delta t}_M)^{k-1}\tv^0+\Delta t\mathcal{H}\left(\modulo{\mathcal{J}}-1+(\mathcal{B}^{\Delta t}_M)^{k-1}\right)\tv^0,\\
    & ~~\vdots\\
    \tv^{kh_M}\leq&(1+\Delta t\tilde{\mathcal{G}})\tv^{kh_M-1}+\Delta t\mathcal{H}\left(\sum_{i\in\mathcal{J}\atop i\ne M}\tv^0+\tv^{(k-1)h_M-1}\right)\\
    \leq&(1+\Delta t\tilde{\mathcal{G}})^{h_M}(\mathcal{B}^{\Delta t}_M)^{k-1}\tv^0+\Delta t\mathcal{H}\left(\modulo{\mathcal{J}}-1+(\mathcal{B}^{\Delta t}_M)^{k-1}\right)\tv^0\sum_{l=0}^{h_M-1}(1+\Delta t\tilde{\mathcal{G}})^l\\
    \leq&(\mathcal{B}^{\Delta t}_M)^{k-1}\mathcal{B}^{\Delta t}_M\tv^0.
\end{align*}
This leads to
$$
\tv^k\leq\left((\modulo{\mathcal{J}}+1)(1+\Delta t\mathcal{M})^{h_M}-\modulo{\mathcal{J}}\right)^{\big\lfloor\frac{ h_{M-1}}{h_M}\big\rfloor-1}\mathcal{B}^{\Delta t}_M\tv^0,\qquad k=0,\dots,\Big\lfloor\frac{h_{M-1}}{h_M}\Big\rfloor h_M,
$$
and also
$$
\tv^k\leq\mathcal{B}^{\Delta t}_{M-1}\mathcal{B}^{\Delta t}_M\tv^0,\qquad k=0,\dots,h_{M-1},
$$
where
$$
\mathcal{B}^{\Delta t}_{M-1}=\left((\modulo{\mathcal{J}}+1)(1+\Delta t\mathcal{M})^{h_{M-1}-\big\lfloor\frac{h_{M-1}}{h_M}\big\rfloor h_M}-\modulo{\mathcal{J}}\right)\left((\modulo{\mathcal{J}}+1)(1+\Delta t\mathcal{M})^{h_M}-\modulo{\mathcal{J}}\right)^{\big\lfloor\frac{ h_{M-1}}{h_M}\big\rfloor-1}.
$$
Observing that $h_i=\tau_i/\Delta t$ for every $i=1,\dots,M$, we can iterate the same argument and pass to the limit as $\Delta t\rightarrow 0$. To avoid unnecessary multiplications, we introduce in the set $\mathcal{J}$ the equivalence relation $\sim$ defined as $i\sim j$ iff $\tau_i=\tau_j$,
and we denote the set of the equivalence classes
\begin{equation}\label{Jtilde}
\tilde{\mathcal{J}}:=\mathcal{J}/\sim.
\end{equation}
Thus, we obtain for $l\in\tilde{\mathcal{J}}$
$$
\sum_{i=1}^M\tv\left(\rho_i^{\Delta x}(k\Delta t,\cdot)\right)\leq\prod_{j\in\tilde{\mathcal{J}}\atop\tau_l\geq\tau_j}\mathcal{B}_j\sum_{i=1}^M\tv(\rho_i^0),\qquad k=0,\dots,h_l,
$$
with
\begin{equation}\label{Bj}
\mathcal{B}_j=
\begin{cases}
(\modulo{\mathcal{J}}+1)\operatorname{exp}\left\{\mathcal{M}\tau_{\min}\right\}-\modulo{\mathcal{J}},
& 
\mbox{ if }\tau_j=\tau_{\min},\\
\begin{aligned}[b]
&\left((\modulo{\mathcal{J}}+1)\operatorname{exp}\Big\{\mathcal{M}\Big(\tau_j-\Big\lfloor\frac{\tau_j}{\tau_{\max}^j}\Big\rfloor\tau_{\max}^j\Big)\Big\}-\modulo{\mathcal{J}}\right)\\
&\qquad\qquad\cdot\left((\modulo{\mathcal{J}}+1)\operatorname{exp}\Big\{\mathcal{M}\tau_{\max}^j\Big\}-\modulo{\mathcal{J}}\right)^{\Big\lfloor\frac{\tau_j}{\tau_{\max}^j}\Big\rfloor-1},
\end{aligned}
&
\mbox{ otherwise },
\end{cases}
\end{equation}
where $\tau_{\min}:=\min_{j\in\tilde{\mathcal{J}}}\tau_j$ and $\tau_{\max}^j:=\max_{l\in\tilde{\mathcal{J}}\atop\tau_j>\tau_l}\tau_l$.
Similarly, we conclude saying that in general \eqref{eq:TVestimate} holds with 
\begin{equation}\label{BT}
\mathcal{B}_T=
\begin{cases}
(\modulo{\mathcal{J}}+1)e^{\mathcal{M}T}-\modulo{\mathcal{J}},
& 
\mbox{ if }T<\tau_{\min},\\
\begin{aligned}[b]
&\Big((\modulo{\mathcal{J}}+1)\operatorname{exp}\Big\{\mathcal{M}\Big(T-\big\lfloor\frac{T}{\tau_{\max}^T}\big\rfloor\tau_{\max}^T\Big)\Big\}-\modulo{\mathcal{J}}\Big)\\
&\qquad\qquad\qquad\cdot\left((\modulo{\mathcal{J}}+1)\operatorname{exp}\Big\{\mathcal{M}\tau_{\max}^T\Big\}-\modulo{\mathcal{J}}\right)^{\big\lfloor\frac{T}{\tau_{\max}^T}\big\rfloor-1},
\end{aligned}
&
\mbox{ otherwise },
\end{cases}
\end{equation}
being $\tau_{\max}^T:=\max_{l\in\tilde{\mathcal{J}}\atop T\geq\tau_l}\tau_l$.
Observe that if $M=1$, omitting the index $i=1$ in the notation and denoting $\int_0^L\omega(s)\d s=:J_0$, then from \eqref{Bj} and \eqref{BT} the constant in the estimate \eqref{eq:TVestimate} translates into
$$
\mathcal{B}_T\prod_{j\in\mathcal{J}\atop T\geq\tau_j}\mathcal{B}_j=\left( 2e^{\mathcal{M}(T-\lfloor T/\tau \rfloor \tau)} -1 \right)
    \left( 2e^{\mathcal{M}\tau} -1 \right)^{\lfloor T/\tau \rfloor},
$$
and this is consistent with the results obtained for the scalar model \cite{CiaramagliaGoatinPuppo2024}.

\section{Detailed proof of the existence theorem}\label{sec:app3}
\begin{proofof}{Theorem~\ref{multiE1}}
By Lemma \ref{multiboundteo} we know that the the approximate solution $\boldsymbol\rho^{\Delta x}$ is uniformly bounded on $[0,T]\times\R $. Moreover, Proposition~\ref{multiBVteo} guarantees that the numerical solution has also uniformly bounded total variation. Thus, from the Helly's Theorem we get that there exists a subsequence of the numerical approximations $\boldsymbol\rho^{\Delta x}$ such that each component $\rho_i^{\Delta x}$ converges in the $\Lloc{1}$-norm to some $\rho_i\in\BV([0,T]\times\R ;[0,R_i])$ as $\Delta x \searrow 0$.
In the following, we apply the classical procedure of Lax-Wendroff theorem to prove that the limit function $\boldsymbol\rho=\left(\rho_1,\dots,\rho_M\right)$ is an entropy weak solution of~\eqref{multiclasse}-\eqref{eq:initial_datum} in the sense of Definition \ref{multientropy}. 
Let $\phi\in\Cc 1([0,T[\,\times\R ;\R^+)$ be a test function and $\phi^n_j=\phi(t_n,x_j)$. We assume that the grid for the approximation is such that $N_T\Delta t<T\leq(N_T+1)\Delta t$. By multiplying \eqref{multidisentropy} by $\Delta x\phi^n_j$ and summing by parts on $n=0,\dots N_T$ and $j\in\mathbb{Z}$, we get
\begin{align}
0\leq&\ \Delta x\sum_j\phi^0_j\modulo{\rho^0_{i,j}-\kappa}+\Delta x\Delta t\sum_{n=1}^{N_T-1}\sum_j\frac{\phi^n_j-\phi^{n-1}_j}{\Delta t}\modulo{\rho^n_{i,j}-\kappa}\label{multipart1}\\
+&\ \Delta x\Delta t\sum_{n=0}^{N_T-1}\sum_j\frac{\phi^n_{j+1}-\phi^n_j}{\Delta x}\left[F^\kappa_{i,j+\frac{1}{2}}(\rho^n_{i,j},\rho^n_{i,j+1})-\sgn(\rho^n_{i,j}-\kappa)\left(\rho^n_{i,j}f_i(\rho^n_{i,j})-\kappa f_i(\kappa)\right)V^{n-h_i}_{i,j}\right]\label{multipart2}\\
+&\ \Delta x\Delta t\sum_{n=0}^{N_T-1}\sum_j\frac{\phi^n_{j+1}-\phi^n_j}{\Delta x}\sgn(\rho^n_{i,j}-\kappa)\left(\rho^n_{i,j}f_i(\rho^n_{i,j})-\kappa f_i(\kappa)\right)V^{n-h_i}_{i,j}\label{multipart3}\\
-&\ \Delta x\Delta t\sum_{n=0}^{N_T-1}\sum_j\sgn(\rho^n_{i,j}-\kappa)\kappa f_i(\kappa)\frac{V^{n-h_i}_{i,j+1}-V^{n-h_i}_{i,j}}{\Delta x}\,\phi^n_j\label{multipart4}\\
-&\ \Delta t\kappa f_i(\kappa)\sum_{n=0}^{N_T-1}\sum_j\left[\sgn(\rho^{n+1}_{i,j}-\kappa)-\sgn(\rho^{n}_{i,j}-\kappa)\right]\left(V^{n-h_i}_{i,j+1}-V^{n-h_i}_{i,j}\right)\, \phi^n_j.\label{multipart5}
\end{align}
Clearly, we have
$$
\eqref{multipart1}\rightarrow\int_\R \modulo{\rho_i^0(x)-\kappa}\phi(0,x)\d x+\int_0^T\int_\R \modulo{\rho_i-\kappa} \del_t\phi \d x\d t,
$$
$$
\eqref{multipart3}\rightarrow\int_0^T\int_\R \sgn(\rho_i-\kappa)\left(\rho_i f_i(\rho_i)-\kappa f_i(\kappa)\right)v_i\left((r\ast\omega_i)(t-\tau_i,x)\right)\del_x\phi \d x\d t,
$$
and
$$
\eqref{multipart4}\rightarrow-\int_0^T\int_\R \sgn(\rho_i-\kappa)\kappa f_i(\kappa)\del_xv_i\left((r\ast\omega)(t-\tau_i,x)\right)\phi\d x\d t,
$$
as $\Delta x\rightarrow 0$. 
Next, we need to prove that both \eqref{multipart2} and \eqref{multipart5} converge to zero.
Let us first focus on~$\eqref{multipart2}$.
 We set $X>0$ such that $\phi(t,x)=0$ for $|x|>X$ and a couple of indexes $j_0,j_1\in\mathbb{Z}$ such that $\phi^n_j=0$ if $j$ is not in $[j_0,j_1]$. Thus,
 \begin{align*}
    |\eqref{multipart2}|&\leq \Delta x\Delta t\norma{\partial_x\phi}\sum_{n=0}^{N_T-1}\sum_{j=j_0}^{j_1}\modulo{F^\kappa_{i,j+\frac{1}{2}}(\rho^n_{i,j},\rho^n_{i,j+1})-\sgn(\rho^n_{i,j}-\kappa) \left(F_i(\rho^n_{i,j})-F_i(\kappa)\right) V_{i,j}^{n-h_i}}\\
    \leq&\ 2\Delta x\Delta t\norma{\partial_x\phi}\left(R_i+\modulo{\kappa}\right)V_i\norma{f_i'}\sum_{n=0}^{N_T-1}\sum_{j=j_0}^{j_1}\modulo{\rho^n_{i,j+1}-\rho^n_{i,j}}+\mathcal{O}(\Delta x)\\
    =&\ \mathcal{O}(\Delta x),
\end{align*}
which follows from the definition of $F^{\kappa}_{i,j+\frac{1}{2}}$, the mean value theorem, the fact that \eqref{multistimamodulo} ensures $V^{n-h_i}_{i,j+1}-V^{n-h_i}_{i,j}=\mathcal{O}(\Delta x)$ and from the bound
\begin{align*}
    \Delta t\sum_{n=0}^{N_T-1}\sum_{j=j_0}^{j_1}\modulo{\rho^n_{i,j+1}-\rho^n_{i,j}}\leq T\sup_{t\in [0,T]} \tv(\rho_i^{\Delta x}(t,\cdot))\leq TC(T,\boldsymbol L,\boldsymbol\tau)\sum_{l=1}^M\tv(\rho_l^0),
\end{align*}
for $C(T,\boldsymbol L,\boldsymbol\tau)$ defined as in \eqref{c}.
\\Finally, we focus on \eqref{multipart5}. Summing again by parts and using that that $V^n_{i,j+1}-V^n_{i,j}=\mathcal{O}(\Delta x)$ holds for all $n\geq-\max_lh_l$, then we get
\begin{align}
    \eqref{multipart5}=&\,\Delta t\kappa f_i(\kappa)\sum_{n=1}^{N_T-1}\sum_j\sgn(\rho^n_{i,j}-\kappa)\left[\left(V^{n-h_i}_{i,j+1}-V^{n-h_i}_{i,j}\right)-\left(V^{n-h_i-1}_{i,j+1}-V^{n-h_i-1}_{i,j}\right)\right]\phi^{n-1}_j\nonumber\\
    &+\mathcal{O}(\Delta x+\Delta t).\label{multistimapart5}
\end{align}
We can write
\begin{align}
    &\left(V^{n-h_i}_{i,j+1}-V^{n-h_i}_{i,j}\right)-\left(V^{n-h_i-1}_{i,j+1}-V^{n-h_i-1}_{i,j}\right)\nonumber\\
=\,&\Delta x~v_i''(\bar{\xi}_{i,j})\left[\xi^{n-h_i}_{i,j}-\xi^{n-h_i-1}_{i,j}\right]\sum_{k=0}^{+\infty}\omega_i^k\left(r^{n-h_i}_{j+k+1}-r^{n-h_i}_{j+k}\right)\label{stimavelocita1}\\
&+\Delta x~v_i'(\xi^{n-h_i-1}_{i,j}) \left[\sum_{k=1}^{N_i}(\omega_i^{k-1}-\omega_i^k)\left(r^{n-h_i}_{j+k}-r^{n-h_i-1}_{j+k}\right)-\omega_i^0\left(r^{n-h_i}_j-r^{n-h_i-1}_j\right)\right]\label{stimavelocita2},
\end{align}
being $\xi^{n-h_i}_{i,j}$ between $\Delta x\sum_{k=0}^{+\infty}\omega_i^kr^{n-h_i}_{j+k}$ and $\Delta x\sum_{k=0}^{+\infty}\omega_i^kr^{n-h_i}_{j+k+1}$, and $\xi^{n-h_i-1}_{i,j}$ between $\Delta x\sum_{k=0}^{+\infty}\omega_i^kr^{n-h_i-1}_{j+k}$ and $\Delta x\sum_{k=0}^{+\infty}\omega_i^kr^{n-h_i-1}_{j+k+1}$, and for $\bar{\xi}_{i,j}$ between $\xi^{n-h_i}_{i,j}$ and $\xi^{n-h_i-1}_{i,j}$.
Regarding \eqref{stimavelocita1}, we remark that 
\begin{align}
    \sum_{k=0}^{+\infty}\omega_i^k\modulo{r^{n-h_i}_{j+k+1}-r^{n-h_i}_{j+k}}&\leq\sum_{l=1}^M\sum_{k=0}^{+\infty}\omega^k_i\modulo{\rho^{n-h_i}_{l,j+k+1}-\rho^{n-h_i}_{l,j+k}}\nonumber\\
&\leq\norma{\omega_i}\sup_{t\in[0,T]}\sum_{l=1}^M\tv(\rho_l^{\Delta x}(t,\cdot))\leq\norma{\omega_i}C(T,\boldsymbol L,\boldsymbol\tau)\sum_{l=1}^M\tv(\rho_l^0),\label{stimavelocitaaux}
\end{align} 
and that $\xi^{n-h_i}_{i,j}-\xi^{n-h_i-1}_{i,j}=\mathcal{O}(\Delta x+\Delta t)$.
Indeed, for some $\theta,\mu\in[0,1]$, we compute
\begin{align*}
\xi^{n-h_i}_{i,j}-\xi^{n-h_i-1}_{i,j}=&\Delta x\sum_{k=0}^{+\infty}\left[\mu\omega_i^kr^{n-h_i}_{j+k+1}+(1-\mu)\omega_i^kr^{n-h_i}_{j+k}-\theta\omega_i^kr^{n-h_i-1}_{j+k+1}-(1-\theta)\omega_i^kr^{n-h_i-1}_{j+k}\right]\\
=&\Delta x\sum_{k=0}^{+\infty}\left[\theta\omega_i^k\left(r^{n-h_i}_{j+k+1}-r^{n-h_i-1}_{j+k+1}\right)+(1-\theta)\omega_i^k\left(r^{n-h_i}_{j+k}-r^{n-h_i-1}_{j+k}\right)\right]\\
&+\Delta x\sum_{k=0}^{+\infty}\left[(\mu-\theta)\omega_i^kr^{n-h_i}_{j+k+1}+[(1-\mu)-(1-\theta)]\omega_i^kr^{n-h_i}_{j+k}\right]\\
=&\Delta x\sum_{k=1}^{N_i}\left(\theta\omega_i^{k-1}+(1-\theta)\omega_i^k\right)\left(r^{n-h_i}_{j+k}-r^{n-h_i-1}_{j+k}\right)\\
&+\Delta x(\mu-\theta)\left[\sum_{k=1}^{+\infty}(\omega_i^{k-1}-\omega_i^k)r^{n-h_i}_{j+k}-\omega_i^0r^{n-h_i}_j\right]\\
&+\Delta x(1-\theta)\omega_i^0\left(r^{n-h_i}_j-r^{n-h_i-1}_j\right).
\end{align*}
Since \eqref{multidelta rho} implies that for every $j\in\mathbb{Z}$
\begin{align}
\modulo{r^{n-h_i}_j-r^{n-h_i-1}_j}\leq& \sum_{l=1}^M\modulo{\rho^{n-h_i}_{l,j}-\rho^{n-h_i-1}_{l,j}}\\
\leq&\lambda\sum_{l=1}^MV_l\left(\modulo{\rho^{n-h_i-1}_{l,j}-\rho^{n-h_i-1}_{l,j-1}}+R_l\norma{f_l'}\modulo{\rho^{n-h_i-1}_{l,j+1}-\rho^{n-h_i-1}_{l,j}}\right)\nonumber\\
&+2\Delta tR\sum_{l=1}^MR_l\norma{v_l'}\omega_l^0,\label{utile}
\end{align}
then, similarly to \eqref{multidelta xi}, we get
\begin{align*}
\big|\xi^{n-h_i}_{i,j}-&\xi^{n-h_i-1}_{i,j}\big|
\leq 2\omega_i^0\Delta x \sum_{k=1}^{N_i}\modulo{r^{n-h_i}_{j+k}-r^{n-h_i-1}_{j+k}}+\Delta x\left(\sum_{k=1}^{N_i}(\omega_i^{k-1}-\omega_i^k)+2\omega_i^0\right)R \\
\leq&2\norma{\omega_i}\lambda\Delta x\sum_{l=1}^MV_l\sum_{k=1}^{N_i}\left(\modulo{\rho^{n-h_i-1}_{l,j+k}-\rho^{n-h_i-1}_{l,j+k-1}}+ R_l\norma{f_l'}\modulo{\rho^{n-h_i-1}_{l,j+k+1}-\rho^{n-h_i-1}_{l,j+k}}\right)\\
&+4\Delta t\norma{\omega_i}R\sum_{k=1}^{N_i}\Delta x\sum_{l=1}^MR_l\norma{v_l'}~\norma{\omega_l} +3\Delta x\norma{\omega_i}R\\
\leq&2\norma{\omega_i}\Delta t\sum_{l=1}^{M}V_l\left(1+R_l\norma{f_l'}\right)\sum_{k\in\Z}\modulo{\rho^{n-h_i-1}_{l,k+1}-\rho^{n-h_i-1}_{l,k}}\\
&+4\Delta t\norma{\omega_i}RL_i\sum_{l=1}^MR_l\norma{v_l'}~\norma{\omega_l}+3\Delta x\norma{\omega_i}R\\
\leq& C_1\Delta x+C_2\Delta t,
\end{align*}
with 
\begin{align*}
C_1&=3\norma{\omega_i}R,\\
C_2&=2\norma{\omega_i}C(T,\boldsymbol L,\boldsymbol\tau)\sum_{l=1}^MV_l\left(1+R_l\norma{f'_l}\right)\sum_{k=1}^M\tv(\rho^0_k)+4\norma{\omega_i}RL_i\sum_{l=1}^MR_l\norma{v_l'}~\norma{\omega_l}.
\end{align*}
Regarding \eqref{stimavelocita2}, for every $j\in\mathbb{Z}$
\begin{align*}
r^{n-h_i}_j-r^{n-h_i-1}_j=&\sum_{l=1}^M\left(\rho^{n-h_i}_{l,j}-\rho^{n-h_i-1}_{l,j}\right)\\
=& \sum_{l=1}^M\left[\lambda\mathcal{R}^n_{l,i,j}\left(\rho^{n-h_i-1}_{l,j+1}-\rho^{n-h_i-1}_{l,j}\right)-\lambda\mathcal{L}^n_{l,i,j}\left(\rho^{n-h_i-1}_{l,j}-\rho^{n-h_i-1}_{l,j-1}\right)\right]+\mathcal{O}(\Delta t),
\end{align*}
 being $\mathcal{R}^n_{l,i,j}=-\rho^{n-h_i-1}_{l,j-1}f'_l(\tilde{\rho}^{n-h_i-1}_{l,j+\frac{1}{2}})V^{n-h_i-1-h_l}_{l,j+1}$ and $\mathcal{L}^n_{l,i,j}=f_l(\rho^{n-h_i-1}_{l,j+1})V^{n-h_i-1-h_l}_{l,j+1}$. Thus, since  $|\mathcal{R}^n_{l,i,j}|\leq R_l\norma{f_l'}V_l$ and $|\mathcal{L}^n_{l,i,j}|\leq V_l$ , and since for every $l=1,\dots,M$ it holds
\begin{align*}
    \lambda\Delta x\Delta t\sum_{k=1}^{N_i}&(\omega_i^{k-1}-\omega_i^k)\sum_{n=1}^{N_T-1}\sum_j\modulo{\rho^{n-h_i-1}_{l,j+k+1}-\rho^{n-h_i-1}_{l,j+k}}\phi^{n-1}_j\\
    &\leq2\lambda\Delta x\Delta t\omega_i^0\norma{\phi}\sum_{n=1}^{N_T-1}\sum_{j=j_0}^{j_1+N_i}\modulo{\rho^{n-h_i-1}_{l,j+1}-\rho^{n-h_i-1}_{l,j}}\\
    &\leq2\lambda\norma{\omega_i}\norma{\phi}\int_0^T\int_{-X}^{X+L_i}\modulo{\rho_l^{\Delta x}\left(t-(h_i+1)\Delta t,x+\Delta x\right)-\rho_l^{\Delta x}\left(t-(h_i+1)\Delta t,x\right)}\d x\d t\\
&\leq2\norma{\omega_i}\norma{\phi}\mathcal{C}\sum_{k=1}^M\tv(\rho_k^0)\Delta t,
\end{align*}
where the positive constant $\mathcal{C}$ is given by Proposition \ref{multiBVteo}, then from \eqref{multistimapart5} we get
\begin{equation*}
\eqref{multipart5}=\mathcal{O}(\Delta x+\Delta t),
\end{equation*}
and this clearly proves that \eqref{multipart5} converges to zero as $\Delta x\rightarrow 0$ ( and $\Delta t\rightarrow 0$).
\end{proofof}

\section*{Acknowledgments}

This work was funded by the European Union’s Horizon Europe research and innovation programme under the Marie Skłodowska-Curie Doctoral Network Datahyking (Grant No. 101072546).
G. Puppo was also supported by the European Union-NextGenerationEU (National Sustainable Mobility Center CN00000023, Italian Ministry of University and Research Decree n. 1033- 17/06/2022, Spoke 9).

{ \small
	\bibliography{nonlocal}
	\bibliographystyle{abbrv}
} 

\end{document}